\documentclass{amsart}
\usepackage{amsmath}
\usepackage{amssymb}
\usepackage[all]{xy}
\usepackage{graphicx}
\usepackage{mathrsfs}
\usepackage{enumitem}

\numberwithin{equation}{section}

\newtheorem{theorem}{Theorem}[section]
\newtheorem{lemma}[theorem]{Lemma}

\newtheorem{definition}[theorem]{Definition}
\newtheorem{example}[theorem]{Example}

\newtheorem{proposition}[theorem]{Proposition}
\newtheorem{corollary}[theorem]{Corollary}

\newtheorem{assumption}[theorem]{Assumption}

\theoremstyle{remark}
\newtheorem{remark}[theorem]{Remark}

\numberwithin{equation}{section}

\newcommand{\arxiv}[1]{\href{http://arxiv.org/abs/#1}{\tt arXiv:\nolinkurl{#1}}}
\newcommand{\arXiv}[1]{\href{http://arxiv.org/abs/#1}{\tt arXiv:\nolinkurl{#1}}}

\usepackage[latin1]{inputenc}
\usepackage{xspace,amssymb,amsfonts,euscript}
\usepackage{amsthm,amsmath}
\usepackage{palatino}
\usepackage{euscript}
\input xy \xyoption {all}

\usepackage{tikz}

\allowdisplaybreaks

\RequirePackage{color}
\definecolor{myred}{rgb}{0.75,0,0}
\definecolor{mygreen}{rgb}{0,0.5,0}
\definecolor{myblue}{rgb}{0,0,0.65}

\RequirePackage{ifpdf}
\ifpdf
  \IfFileExists{pdfsync.sty}{\RequirePackage{pdfsync}}{}
  \RequirePackage[pdftex,
   colorlinks = true,
   urlcolor = myblue, % \href{...}{...} external (URL)
   citecolor = mygreen, % \cite{...}
   linkcolor = myred, % \ref{...} and \ageref{...}
   pagebackref,
   bookmarksopen=true]{hyperref}
\else
  \RequirePackage[hypertex]{hyperref}
\fi

\makeatletter
\def\@tocline#1#2#3#4#5#6#7{\relax
  \ifnum #1>\c@tocdepth % then omit
  \else
    \par \addpenalty\@secpenalty\addvspace{#2}%
    \begingroup \hyphenpenalty\@M
    \@ifempty{#4}{%
      \@tempdima\csname r@tocindent\number#1\endcsname\relax
    }{%
      \@tempdima#4\relax
    }%
    \parindent\z@ \leftskip#3\relax \advance\leftskip\@tempdima\relax
    \rightskip\@pnumwidth plus4em \parfillskip-\@pnumwidth
    #5\leavevmode\hskip-\@tempdima
      \ifcase #1
       \or\or \hskip 1em \or \hskip 2em \else \hskip 3em \fi%
      #6\nobreak\relax
    \dotfill\hbox to\@pnumwidth{\@tocpagenum{#7}}\par
    \nobreak
    \endgroup
  \fi}
\makeatother

\newcommand{\nc}{\newcommand}

\nc{\flags}{\mathcal{F}}

\def\wt{\text{wt}}

\def\ii{{\bf i}}

\def\N{\mathbb{N}}

\def\C{\mathcal{C}}
\def\R{\mathbb{R}}
\def\Z{\mathbb{Z}}
\def\F{\mathcal{F}}

\def\B{\mathcal{B}}
\def\op{{\rm{op}}}

\def\jj{{\bf j}}

\def\a{\alpha}
\def\b{\beta}

\def\la{\lambda}

\def\ga{\gamma}
\def\Ga{\Gamma}
\def\d{\delta}

\def\w{\omega}

\def\e{\epsilon}

\def\inv{^{-1}}

\def\uj{\underline{\smash{j}}}

 \def\coker{\operatorname{coker}}

 \def\hd{\operatorname{hd}}

\def\G{\mathcal{F}}

\def\D{\Delta}

\def\Hom{\,\mbox{Hom}}
\def\dimq{\dim_q}
\def\Ind{\operatorname{Ind}}
\def\Coind{\operatorname{CoInd}}
\def\seq{\operatorname{Seq}}
\def\dual{^\circledast}

\def\End{\operatorname{End}}
\def\Seq{\operatorname{Seq}}

\def\Res{\operatorname{Res}}

\def\Ext{\operatorname{Ext}}

\def\mods{\mbox{-mod}}

\def\id{\mbox{id}}

\newcommand{\map}[2]{\,{:}\,#1\!\longrightarrow\!#2}
\newcommand{\spann}{\operatorname{span}}

\def\inv{^{-1}}

\numberwithin{equation}{section}

\title{Face Functors for KLR Algebras}

\author[Peter J McNamara]{Peter J McNamara}
\address{School of Mathematics and Physics,
University of Queensland, St Lucia, QLD, Australia.}
\email{maths@petermc.net}

\author[Peter Tingley]{Peter Tingley}
\address{Department of Mathematics and Statistics, Loyola University, Chicago, IL, USA.}
\email{ptingley@luc.edu}


\date{\today}

\begin{document}

\begin{abstract}
Simple representations of KLR algebras can be used to realize the infinity crystal for the corresponding symmetrizable Kac-Moody algebra. It was recently shown that, in finite and affine types, certain sub-categories of ``cuspidal" representations realize crystals for sub Kac-Moody algebras. Here we put that observation an a firmer categorical footing by exhibiting a corresponding functor between the category of representations of the KLR algebra for the sub Kac-Moody algebra and the category of cuspidal representations of the original KLR algebra. 
\end{abstract}

\maketitle

\tableofcontents

\section{Introduction}

Khovanov-Lauda-Rouquier (KLR) algebras have been a subject of intense study in the last few years, due to the fact that their categories of graded representations categorify various objects in the theory of quantum groups. Most importantly here, the gradable simple modules for a KLR algebra can be used to realize the crystal $B(-\infty)$ of the positive part of a universal enveloping algebra. 

In \cite{tingleywebster}, the correspondence between simple modules for KLR algebras and crystals is used to define Mirkovi\'c-Vilonen polytopes in all affine types. There the polytopes arise as the character polytopes of representations of KLR algebras (along with some decoration in affine types). One of the key observations is that, for every face of the MV polytope, there was a sub-category (the cuspidal representations for that face) whose simple representations realized a crystal for a lower rank enveloping algebra. There the correspondence is just combinatorial, but it is natural to ask for a corresponding functor from representations of the smaller rank KLR algebra to representations of the original KLR algebra. 
Here we construct such a functor under the following conditions:

\begin{assumption} \label{assumptionA} Either
% \begin{enumerate}
% \item The original root system is of finite type, or
% 
% \item The original root system is of symmetric affine type, and the KLR algebra is defined over a field of characteristic $0$.  {\color{blue} should this say geometric?} {\color{orange}Yes}
% \end{enumerate}
% 
\begin{itemize}
\item The original root system is of finite type, or

\item The original root system is of symmetric affine type, and the KLR algebra is geometric and over a field of characteristic $0$.
\end{itemize}
\end{assumption}
\noindent In fact, our construction works under somewhat more general conditions, see Assumption \S\ref{assumptionB}. 
%, with the exception of \S \ref{sub:standard} and \S\ref{sub:crystal} which require more restrictived hypotheses, though still including the two cases in the above bullet points.
We construct the functor by considering a certain projective object $P$ in the category of cuspidal representations which has the property that $\End(P)$ is isomorphic to a new KLR algebra. Our face functor is then $X\rightarrow P \otimes_{\End(P)} X$. 

If the original KLR algebra satisfies Assumption \ref{assumptionA} and the face is of finite type then $P$ is a projective generator and our functor is an equivalence of categories. More generally this is not true, but we still show that the functor respects much of the structure. In particular, it intertwines the crystal operators for the KLR algebra associated to the face and the crystal operators from \cite{tingleywebster} for the category of cuspidal modules corresponding to the face whenever both are defined. % (see Corollary \ref{cor:fcrystal}). 

%In the final section, \S\ref{sec:imcat}, we 
We then analyze some implications of our construction for the category of imaginary cuspidal representations of $R(\delta)$ in the affine case. Essentially, our functors allow us to understand all such categories in terms of the $\widehat{\mathfrak{sl}}_3$ case. This allows us to show that the category of cuspidal modules of $R(\delta)$ is equivalent to the category of representations of $k[z] \otimes Z$ where $Z$ is the finite type zig-zag algebra defined in \cite{HK}. This last result in the special case of a balanced convex order appears in \cite{KM}, where they also consider the cuspidal representations of $R(n\delta)$ for all $n$. Their proof relies on significantly more case-by-case analysis, which our reduction to $\widehat{\mathfrak{sl}}_3$ avoids. We believe our arguments should generalize to give their result on cuspidal $R(n\delta)$-modules, which would generalize it to arbitrary convex orders.

Recent work of Kashiwara and Park \cite{kashiwarapark} is closely related to our construction. 
Our face functors are examples of the Kashiwara-Park functors, at least in finite type.
Kashiwara and Park do not show that their functors are equivalences of categories for finite type faces (which is not true in the generality they consider), and do not see the connection with the crystal structure in \cite{tingleywebster}. 

\subsection{Acknowledgements} 
The ideas for this paper arose during the workshop ``Algebraic Lie theory and representation theory" at ICMS in Scotland, September 2014. We thank the organizers of that event for a great meeting. We also thank the anonymous referees for many insightful suggestions. P.M. was partially supported by ARC grant DE150101415. 
P.T. was partially supported by NSF grant DMS-1265555.

\section{Preliminaries}

\subsection{Root systems and convex orders} \label{ss:co}

Fix a root system $\Phi$. %, which we assume is either of finite type or symmetric affine type. 
Let 
$(a_{ij})_{i,j\in I}$ be its Cartan matrix, and $d_i$ the usual symmetrizing factors. 
Let $\Delta$ denote the set of simple roots of $\Phi$. For $i,j\in I$, write $i\cdot j$ for $d_i a_{ij}$. Then $i\cdot j=j\cdot i$. Extend this by linearity to a $\Z$-valued bilinear pairing on $\N I$. For $\lambda = (n_i)_{i \in I }\in \N I$, define $ \text{ht}(\lambda)= \sum_{i \in I} n_i$.   

\begin{definition}\label{facedefinition}
 A face is a decomposition of $\Phi^+$ into three disjoint subsets
 \[
  \Phi^+=F^+\sqcup F\sqcup F^-
 \]
such that, for all $x\in\spann_{\R_{\geq 0}}F$:
\begin{enumerate}
\item If $y\in \spann_{\R_{\geq 0}}F^+$ is non-zero, then  $x+y\notin \spann_{\R_{\geq 0}}F^- \cup F$.
 \item If $y\in \spann_{\R_{\geq 0}}F^-$ is non-zero, then  $x+y\notin \spann_{\R_{\geq 0}}F^+ \cup F$.
\end{enumerate}
We often denote a face simply by $F$, with $F^+, F^-$ being suppressed. 
\end{definition}

\begin{remark}
The terminology ``face" comes from the fact that, in finite type, there is a bijection between faces in the sense of Definition \ref{facedefinition} and faces of the Weyl polytope (i.e. the convex hull of the Weyl group orbit of $\rho$). %A related statement holds in all types, as discussed in \cite{tingleywebster}. {\color{orange}where?}
\end{remark}

\begin{definition} \label{def:convex}  \cite[Definition 1.8]{tingleywebster}
 A {\bf convex preorder} is a pre-order $\succ$ on $\Phi^+$
  such that, for every equivalence class $C$, 
  $F^- = \{ \alpha : \alpha \prec C\}, F=C, F^+=   \{ \alpha : \alpha \succ C\}$ is a face. A {\bf convex order} is a convex pre-order which is a total order on positive real roots. 
\end{definition}

\begin{remark} 
There is also a notion of convex order based on the condition that if three roots satisfy $\alpha+\beta=\gamma$, then $\gamma$ is between $\alpha$ and $\beta$. It is clear from definitions that any order which is convex according to Definition \ref{def:convex} is also convex in that sense. In finite type the two notions were shown to be equivalent in \cite[Paragraph before Proposition 1.16]{tingleywebster}. In affine type the equivalence is shown in \cite[Theorem 3.2]{mcn3}.
\end{remark}

\begin{definition}
A convex order $\prec$ is said to be compatible with a face $F$ if $F$ is an interval for $\prec$ and $F^- \prec F \prec F^+$. 
\end{definition}

For any linear functional $c: {\Bbb R} \Phi \rightarrow {\Bbb R}$, we get a convex pre-order on $\Phi^+$ by setting 
 $\alpha \preceq \beta$ if $c(\alpha)/\text{ht}( \alpha)\leq c(\beta)/\text{ht}( \beta)$. The following should be considered a partial converse to this statement. 

\begin{lemma}\cite[Lemma 1.10]{tingleywebster} \label{lem:co-hyper}
For any face $F$ there is a sequence $\{H_n\}_{n\in\N}$ of cooriented hyperplanes in $\R\Phi^+$ such that
\begin{itemize}
 \item $F\subset H_n$ for all $n$,
 \item for all $\a\in F^+$, $\a$ lies on the positive side of $H_n$ for $n\gg 0$, and
 \item for all $\a\in F^-$, $\a$ lies on the negative side of $H_n$ for $n\gg 0$.
\end{itemize}
In particular, for any $N$, if we let $\Gamma$ be the set of roots in $\Phi^+$ of height at most $N$, 
%\end{lemma}
%
%Lemma \ref{lem:co-hyper} implies that, for any finite $\Gamma \subset \Phi^+$, 
there is a linear functional $c\map{\R\Phi^+}{\R}$ such that 
\[
 (F^+ \cap \Gamma ,F \cap \Gamma ,F^- \cap \Gamma)=(c\inv(\R_{>0}) \cap \Gamma, c\inv(0)\cap  \Gamma,c\inv(\R_{<0})\cap  \Gamma).
\]
\end{lemma}

Let $\D_F$ be the set of positive real roots in $F$ which cannot be written as sums of other positive roots in $F$. Let $\Phi_F$ be the corresponding root system whose simple roots are $\D_F$. 
If $\Phi$ is at worst affine then, as discussed in \cite[\S3.2]{tingleywebster}, $\Phi_F$ is a product of finite and affine root systems (although imaginary root spaces may decompose). 

% We call the corresponding root system $\Phi_F$ and let $\Delta_F$ be the sets of simple roots. Explicitly, $\Delta_F$ is the set of positive real roots in $F$ which cannot be written as sums of other positive roots in $F$. 

\begin{remark} The parenthetical remark above is because, if $\Phi_F$ has two or more affine components, it contains non-parallel imaginary roots, all of which come from the imaginary root space of $\Phi$. See \cite[Remark 3.16]{tingleywebster} for an example. 
\end{remark}

\begin{lemma} \label{lem:order-exists}
For every face $F$ there is a compatible convex order $\prec$. Furthermore, for any convex order $\prec_F$ on $\Phi_F^+$, we can choose $\prec$ such that its restriction to ${\Phi_F^+} $ agrees with $ \prec_F$. 
\end{lemma}

\begin{proof}
Choose a sequence of linear functionals $\Gamma_n$ which define the co-oriented hyperplanes from Lemma \ref{lem:co-hyper}.
Define $\alpha \leq \beta$ if $\Gamma_n(\alpha)/\text{ht}( \alpha)\leq \Gamma_n(\beta)/\text{ht}( \beta)$ for all sufficiently large $n$. This is manifestly a convex pre-order, and $\Phi_F^+$ is an equivalence class. By \cite[Lemma 1.14]{tingleywebster}, this can be refined to a convex total order, which can be made to agree with any chosen convex order on $\Phi_F^+$.
\end{proof}

%%%%%%%%%%%%%%%%%%%%%%%%%%%%%%%%%%%%%%%%%%%%%%%%%%%%%%%%
\subsection{KLR algebras} \label{ss:KLR}
%%%%%%%%%%%%%%%%%%%%%%%%%%%%%%%%%%%%%%%%%%%%%%%%%%%%%%%%%%%%%%
Here we review the Khovanov-Lauda-Rouquier algebras introduced in \cite{khovanovlauda, rouquier}.
For any $\nu\in\N I$, define
\[
 \Seq(\nu)=\{\ii=(i_1,\ldots,i_{\text{ht}(\nu)})\in I^{\text{ht}(\nu)}\mid\sum_{j=1}^{\; \text{ht}(\nu)} i_j=\nu\}.
\]
The symmetric group $S_{\text{ht}(\nu)}$ acts on  $\Seq(\nu)$. Let $s_i$ denote the adjacent transposition $(i,i+1)$.
 Fix a field $k$. %which we assume to be of characteristic zero if $\Phi$ is of affine type.
For each $i,j\in I$, choose $Q_{ij}(u,v)\in k[u,v]$ such that
\begin{enumerate}[label=(Q\arabic*)]
\item  \label{eq:Qcond1} $Q_{ii}(u,v)=0.$

\item \label{eq:Qcond2} 
If $u$ has degree $2d_i$ and $v$ has degree $2d_j$ then $Q_{ij}$ is homogeneous of degree $-2d_i a_{ij}=-2d_j a_{ji}$. Furthermore, the coefficients $t_{ij}$ and $t_{ji}$ of $u^{-a_{ij}}$ and $v^{-a_{ji}}$ are both nonzero.
 
\item \label{eq:Qcond3} 
$ Q_{ij}(u,v)=Q_{ji}(v,u).$
\end{enumerate}

\begin{definition} \label{def:KLR}
The KLR algebra $R(\nu)$ is the associative $k$-algebra generated by elements $e_{\bf i}$, $y_j$, $\psi_k$ with ${\bf i}\in \seq(\nu)$, $1\leq j\leq \text{ht}(\nu)$ and $1\leq k< \text{ht}(\nu)$, subject to the relations
\begin{equation} \label{eq:KLR}
\begin{aligned}
& e_\ii e_{\jj} = \delta_{\ii, \jj} e_\ii, \ \ 
\sum_{\ii \in {\rm Seq}(\nu)}  e_\ii = 1, \ \
 y_{k} y_{l} = y_{l} y_{k}, \ \ y_{k} e_\ii = e_\ii y_{k}, \\
& \psi_{l} e_\ii = e_{s_{l}\ii} \psi_{l}, \ \ \psi_{k} \psi_{l} =
\psi_{l} \psi_{k} \ \ \text{if} \ |k-l|>1, \ \
 \psi_{k}^2 e_\ii = Q_{i_{k}, i_{k+1}} (y_{k}, y_{k+1})e_\ii, \\
& (\psi_{k} y_{l} - y_{s_k(l)} \psi_{k}) e_\ii =\begin{cases}
-e_\ii \ \ & \text{if} \ l=k, i_{k} = i_{k+1}, \\
e_\ii \ \ & \text{if} \ l=k+1, i_{k}=i_{k+1}, \\
0 \ \ & \text{otherwise},
\end{cases} \\[.5ex]
& (\psi_{k+1} \psi_{k} \psi_{k+1}-\psi_{k} \psi_{k+1} \psi_{k}) e_\ii\\
& =\begin{cases} \dfrac{Q_{i_{k}, i_{k+1}}(y_{k},
y_{k+1}) - Q_{i_{k}, i_{k+1}}(y_{k+2},  y_{k+1})}
{y_{k} - y_{k+2}}e_\ii\ \ & \text{if} \
i_{k} = i_{k+2}, \\
0 \ \ &  \text{otherwise}.
\end{cases}
\end{aligned}
\end{equation}
Define $R=\displaystyle\bigoplus_{\nu\in\N I} R(\nu)$.
\end{definition}

This is a graded algebra where each $e_\ii$ has degree zero, $y_j e_\ii$ has degree $i_j\cdot i_j$ and $\psi_k e_\ii$ has degree $-i_k\cdot i_{k+1}$. Fix $\nu$. For each $\sigma \in S_{\text{ht}(\nu)}$, choose a reduced expression $\sigma = s_{i_1} \cdots s_{i_k}$, and let $\psi_{\sigma} = \psi_{i_1} \cdots \psi_{i_k}$ As shown in \cite[2.5]{khovanovlauda},
\begin{equation*} \label{eq:basis}
\left\{ \psi_\sigma \left( \prod_{k=1}^{\text{ht}(\nu)} y_k^{r_k}  \right) e_{{\bf i}} \mid {\bf i} \in  \Seq(\nu), r_1, \ldots, r_{\text{ht}(\nu)} \geq 0, \sigma \in S_{\text{ht}(\nu)} \right\}
\end{equation*}
is a basis for $R(\nu)$.

Throughout this paper we assume all modules are graded and finitely generated. 
We write $q$ for the grading shift functor, if $V=\oplus_{n\in \Z} V_n$, then $(qV)_n=V_{n-1}$. For modules $U$ and $V$, the space of homomorphisms $\Hom(U,V)$ is graded, we have $\Hom(U,V)=\oplus_{n\in\Z} \Hom(U,V)_n$, where $\Hom(U,V)_n$ is the space of homogenous homomorphisms from $q^nU$ to $V$, or equivalently from $U$ to $q^{-n}V$.
For a graded vector space $V=\oplus_{n\in \Z} V_n$, its graded dimension is $\dimq V=\sum_{n\in\Z}(\dim V_n) q^n$.
If $M$ is an $R(\nu)$ module we say $\wt(M)=\nu$, and define its character to be the formal sum
\[
 \operatorname{ch}(M)=\sum_\ii \dim_q (e_\ii M) [\ii].
\]
Since we only consider finitely generated modules these dimensions are all finite.
 
\begin{remark}
There is also a diagrammatic approach to $R$, see \cite{khovanovlauda,kl2,tingleywebster}. There the $y_i$ are represented by dots and the $\psi_i$ by crossings.
\end{remark}

Fix $\la,\mu\in\N I$. 
There is a natural inclusion $\iota_{\la,\mu}:R(\la)\otimes R(\mu)\to R(\la+\mu)$ defined by $\iota_{\la,\mu}(e_\ii\otimes e_\jj)=e_{\ii\jj}$, $\iota_{\la,\mu}(y_i\otimes 1)=y_i$, $\iota_{\la,\mu}(1\otimes y_i)=y_{i+\text{ht}(\la)}$, $\iota_{\la,\mu}(\psi_i\otimes 1)=\psi_i$, $\iota_{\la,\mu}(1\otimes \psi_i)=\psi_{i+\text{ht}(\la)}$. These combine to give an inclusion $\iota: R \otimes R \to R$. 
 Let $e_{\la\mu}$ be the image of the unit under the inclusion $R(\la)\otimes R(\mu)\to R(\la+\mu)$

\begin{definition} \label{def:IR}
The induction functor $\Ind_{\la,\mu}:R(\la)\otimes R(\mu)\mods \to R(\la+\mu)\mods$ is given by
\[
 \Ind_{\la,\mu} (M) = R(\la+\mu)e_{\la\mu} \bigotimes_{R(\la)\otimes R(\mu)} M.
\]
For an $R(\la)$-module $A$ and an $R(\mu)$-module $B$, we write $A\circ B$ for $\Ind_{\la,\mu}(A\otimes B)$.

The restriction functor $\Res_{\la,\mu}:R(\la+\mu)\mods \to R(\la)\otimes R(\mu)\mods$ is given by
\[
 \Res_{\la,\mu}(M) = e_{\la\mu} M.
\]
 
 \end{definition}
By iterating Definition \ref{def:IR}, functors such as 
$$ \Res_{\lambda_1,\ldots,\lambda_l}: R(\lambda_1,\ldots,\lambda_l)\text{-mod} \rightarrow R(\lambda_1) \otimes \cdots \otimes R(\lambda_l)\text{-mod}$$ 
$$\Ind_{\la_1,\ldots,\la_k}:  R(\lambda_1) \otimes \cdots \otimes R(\lambda_l)\text{-mod} \rightarrow R(\lambda_1,\ldots,\lambda_l)\text{-mod}$$
can be unambiguously defined. 

The induction functor is left adjoint to the restriction functor. Since $\Ind_{\la,\mu} $ sends projective modules to projective modules, there is a natural isomorphism
\begin{equation}\label{extadjunction}
 \Ext^i(M_1\circ\cdots\circ M_n,N)\cong \Ext^i(M_1\otimes\cdots\otimes M_n,\Res N).
\end{equation}

The restriction functor also has a right adjoint, given by 
the usual coinduction construction. We will need \cite[Theorem 2.2]{laudavazirani}, which says
\begin{equation}
 \Coind(A\otimes B)\cong q^{-\la\cdot\mu} B\circ A.
\end{equation}
 In \cite{laudavazirani} this is stated for finite dimensional modules, but the same proof shows it is true in complete generality.
Thus, if $A$ is a $R(\la+\mu)$-module, $B$ is a $R(\la)$-module and $C$ is a $R(\mu)$-module, then there is a natural isomorphism
\begin{equation}\label{otheradjunction}
 \Hom(\Res A, B\otimes C)\cong \Hom(A,q^{-\la\cdot\mu}C\circ B).
\end{equation}

The following is often called the Mackey filtration.

\begin{theorem}  \cite[Proposition 2.18]{khovanovlauda}\label{mackey} 
Let $\la_1,\ldots,\la_k,\mu_1\ldots,\mu_l\in\N I$ be such that $\sum_i \la_i=\sum_j \mu_j$, and let $M$ be a $R(\la_1)\otimes \cdots \otimes R(\la_k)$-module. Then
the composition
$$ \Res_{\mu_1,\ldots,\mu_l}\circ\Ind_{\la_1,\ldots,\la_k}(M)$$ 
has a filtration indexed by tuples $\eta_{ij}$ satisfying $\la_i=\sum_j \eta_{ij}$ and $\mu_j=\sum_i\eta_{ij}$.
Every subquotient of this filtration is isomorphic to $\Ind_\eta^\mu\circ \tau \circ \Res_\eta^\la(M)$, where\begin{itemize}
\item $\Res_{\eta}^\la\map{\otimes_i R(\la_i)\mods}{\otimes_i(\otimes_jR(\eta_{ij}))\mods}$
is the tensor product of the $\Res_{\eta_{i\bullet}},$ 
\item $\tau\map{\otimes_i(\otimes_jR(\eta_{ij}))\mods}{\otimes_j(\otimes_iR(\eta_{ij}))\mods}$ is given by permuting the tensor factors and shifting the degree. The degree shift is the degree of the permutation acting on the idempotent $e_{\eta_{11} \eta_{12} \cdots \eta_{  k \ell}},$ and
\item $\Ind_\eta^\mu\map{\otimes_j(\otimes_iR(\eta_{ij}))\mods}{\otimes_j R(\mu_j)\mods}$ is the tensor product of the $\Ind_{\eta_{\bullet i}}$.
\end{itemize}
\end{theorem}

There is an anti-automorphism $\dagger$ of $R$ which fixes each of the generators $e_\ii$, $y_j$ and $\psi_k$. Thus, given a finite dimensional $R$-module $M$, we can define the structure of an $R$-module on its dual $M\dual=\Hom_k(M,k)$ by $(r\cdot \phi)(m)=\phi(r^\dagger \cdot m)$. As in \cite[\S 3.2]{khovanovlauda} each finite dimensional simple module is isomorphic to its dual up to a grading shift. The dual induces the bar involution on the Grothendieck group. Its behavior with respect to induction is
\begin{equation}\label{eq:dualind}
 (M\circ N)\dual \cong q^{\wt(M)\cdot \wt(N)} N\dual\circ M\dual.
\end{equation}

As discussed in \cite{kl2},
different choices of $Q_{ij}$ can produce isomorphic algebras. 

\begin{definition}
Let $R$ be a KLR algebra of symmetric type. 
We say  $R$ is of geometric type if $R$ is isomorphic to the one defined in \cite{mcn3} via an isomorphism $\phi$ which respects the inclusion $\iota: R \otimes R \to R$ in the sense that  $\phi \circ \iota = \iota \circ (\phi \otimes \phi)$.  
\end{definition}

The terminology is due to a geometric interpretation of KLR algebras with this choice of parameters.
The following is discussed by Khovanov and Lauda \cite{kl2}.

\begin{lemma}  $\mbox{}$
\label{lem:geom-characterization}
\begin{enumerate}

\item \label{ar1} Every KLR algebra whose Dynkin diagram is a tree is of geometric type. In particular, every KLR algebra of affine type $D_n$ and $E_n$ is of geometric type. 

\item \label{ar2} For $n \geq 3$, A KLR algebra of type $\widehat{\mathfrak{sl}}_n$ with $I={\Bbb Z}/n {\Bbb Z}$ and $Q_{i,i+1}= s_i u+ t_i v$ is of geometric type if and only if $s_1 \cdots s_n/ t_1 \cdots t_n=(-1)^n$. 

\item \label{ar3} A KLR algebra of type $\widehat{\mathfrak{sl}}_2$ is of geometric type if and only if the quadratic polynomial $Q_{0,1}(u,v)$ has discriminant zero.
\end{enumerate}
\end{lemma}

\begin{proof} Fix a root system. 
There is an action of $({k}^*)^{I \times I}$ on the set of KLR algebras where $(z_{ij})_{i,j \in I}$ acts by $\psi_{j} e_{\bf i} \mapsto z_{i_j i_{j+1}} \psi_{j} e_{\bf i}$, $y_j e_{\bf i}\mapsto z_{i_j i_j}^{-1} y_j e_{\bf i} $ and $e_{\bf i} \mapsto e_{\bf i}$. This sends algebras to isomorphic algebras. 
The $Q_{ij}$ change according to $Q'_{ij}(u,v)= z_{ij} z_{ji} Q_{ij} (z_{ii} u, z_{jj} v)$.
Furthermore, all possible isomorphisms of KLR algebras that fix the idempotents $e_{\bf i}$ arise in this way. To see this notice that, for $i \neq j$, looking at only 2 strands, it is clear for weight reasons that any automorphism must send $\psi_1e_\ii$ to a multiple of itself, say $z_{i_1i_2} \psi_1e_\ii$. Also, looking at a single strand, $y_1$ must be sent to a multiple of itself, say $z_{i_1i_1} y_1$. The relation that, if $i_1=i_2,$  $(\psi_1 y_2- y_1 \psi_1)e_\ii= e_{\bf i}$ implies that a crossing with both stands colored $i$ must be scaled by $z_{ii}^{-1}$. 

If the Dynkin diagram is a tree, it is clear that all choices of $Q_{ij}$ are related by such isomorphisms. This establishes \eqref{ar1}.  

In type $\widehat{\mathfrak{sl}}_n$, for the KLR algebras used in \cite{mcn3}, it is easy to check that the stated relations hold. The action of $({k}^*)^{I \times I}$ preserves these relations, as do diagram automorphisms, so they hold for all geometric type KLR algebras. Furthermore, one can check that all KLR algebras that satisfy these conditions are related in this way, establishing \eqref{ar2} and \eqref{ar3}. 
\end{proof}

\begin{definition}
A pseudo-KLR algebra is a KLR algebra in the sense of Definition \ref{def:KLR} with the requirement that $t_{ij}\neq 0$ in Condition \ref{eq:Qcond2} removed.
\end{definition}

The following demonstrates that pseudo-KLR algebras which are not in fact KLR algebras have noticeably different categories of representations. Fix $i \neq j$. 
Due to the categorification results from \cite{khovanovlauda}, for an actual KLR algebra $R((1-a_{ij}) i+j)$, the number of irreducible modules up to grading shift is exactly the dimension of that weight space of $U^+$, so $1-a_{ij}$. 

\begin{theorem}\label{pseudo}
In a pseudo-KLR algebra with the coefficient $t_{ij}=0$ for some $i \neq j$, the algebra $R((1-a_{ij}) i+ j)$ has at least $2-a_{ij}$ irreducible modules.
\end{theorem}

\begin{proof}
For each $a,b \in {\Bbb Z}_{\geq 0}$ with $a+b=1-a_{ij}$, consider the irreducible $R(ai)\otimes R(j)\otimes R(bi)$-module $L_{a,b}= L(i^a)\otimes L(j)\otimes L(i^b)$. We claim that $L_{a,b}$ is in fact a module for $R((1-a_{ij})i+j)$, where all generators which are not in $R(ai)\otimes R(j)\otimes R(bi)$ act by zero. To see this, one must check that the relations defining the KLR algebra are all respected. This happens because, since $t_{ij}=0$, the polynomial $Q_{ij}(u,v)$ is divisible by $v$, which corresponds to a dot $y_{a+1}$ on the $j$ colored strand, and this dot kills the simple $L(j)$. The right side of the difficult relations (the ones involving multiple $\psi$'s and which involve the $j$ strand) then all have a factor of $y_{a+1}$.

Each $L_{a,b}$ is irreducible since its restriction to $R(ai)\otimes R(j)\otimes R(bi)$ is irreducible. Thus we have found $2-a_{ij}$ distinct irreducible modules.
\end{proof}

\subsection{Cuspidal modules} \label{ss:cuspidal} 
For the remainder of this section, fix a face $F$ and a compatible convex order $\prec$. 
An R-module $M$ is called $F$-cuspidal if, for all $\la,\mu\in\N I$ such that $\Res_{\la,\mu}M\neq 0$,
$$\la\in \spann_{\R_{\geq 0}}(F^-\cup F) \quad \text{ and } \quad \mu\in \spann_{\R_{\geq 0}}(F^+\cup F).$$ 

Fix a module $M$ and $c$ as in Lemma \ref{lem:co-hyper} for $N=  \text{ht}(\wt(M))$.
%, 
%$\beta$ is in $F^+/F/F^-$ if and only of $c(\beta)$ is positive/zero/negative.
In \cite{tingleywebster}, $M$ is called semi-cuspidal if
$\Res_{\la,\mu}M\neq 0$ implies $c(\la)\leq0$ (in that paper there is a stronger notion of cuspidal, which is why the term semi-cuspidal is used). This notion of semi-cuspidal is equivalent to our notion of cuspidal, but to see this one must use \cite[Theorem 2.4 and Corollary 2.12]{tingleywebster}, which imply that, if $\Res_{\la,\mu}M\neq 0$ and $c(\la)>0$, then there is a root $\beta$ with $\Res_{\beta,\wt(\lambda)-\beta}M\neq 0$ and $c(\beta)>0$. In any case, we see:

\begin{lemma} \label{lemma:other-cuspidal}
Fix a module $M$ and $c$ as in Lemma \ref{lem:co-hyper} such that, for roots $\beta$ of height at most $\wt(M)$, 
$\beta$ is in $F^+/F/F^-$ if and only of $c(\beta)$ is positive/zero/negative.
Then $M$ is cuspidal if and only if $\Res_{\la,\mu}M\neq 0$ implies $c(\la)\leq0$.
\qed
\end{lemma}

%
%
%\begin{remark}
%In \cite{tingleywebster} there is a stronger notion of cuspidal, which is why there the term semi-cuspidal is used.
%\end{remark}

\begin{definition}
Let $\C_F$ denote the full subcategory of cuspidal R-modules. 
\end{definition}

%For any $m\in \N$, 
%Lemma \ref{lem:co-hyper} gives a linear functional $c$ such that, for roots $\beta$ of height at most $m$, 
%$\beta$ is in $F^+/F/F^-$ if and only of $c(\beta)$ is positive/zero/negative. The following is essentially contained in \cite[\S 2]{tingleywebster}. 

%\begin{lemma}\label{lem:2.19}
%For any $M$ with $\wt(M) \in  \spann_{\R\geq 0}(F)$ and $c$ chosen as above for $m=\text{ht}(\wt(M))$, 
%$M \in \C_F$ if and only if, whenever $\Res_{\la,\mu}M\neq 0$, we have $c(\lambda)\leq 0$. 
%\end{lemma}
%
%\begin{proof} 
%The ``if" statement is obvious. So, assume $M$ is such that, for some $\la$, $\Res_{\la,\mu}M\neq 0$, and $c(\lambda)>0$. Choose some irreducible subquotent of $\Res_{\la,\mu}M$, which must be of the form $L_1 \otimes L_2$, where $L_1,L_2$ are irredcuible for $R(\lambda), R(\mu)$ respectively. By \cite[Theorem 2.4]{tingleywebster}, 
%$$L_1 = \hd (L_{11} \circ L_{12} \circ \cdots L_{1k})$$
%for some irredicible $L_{1j}$ with
%$c(\wt(L_{11})) > c(\wt(L_{12})) > \cdots > c(\wt(L_{1k}))$ ***not correct***, and each
%$\wt(L_{1k})$ is a multiple of a root $\beta_{1k}.$ 
%Furthermore, by \cite[Corollary 2.12]{tingleywebster}, each $\wt(L_{ik})$ is a positive linear compbination of roots, all of which have the same value of $c(\beta)/\text{ht}(\beta)$; in particular, $c$ has a constant sign on each of these sets of roots. 
%Since $c(\lambda)>0$, we must certainly have $c(\wt(L_{1k}))>0$. But 
%$\Res_{\wt(L_{11}), \wt(M)- \wt(L_{11})}M \neq 0$, so $M$ is not cuspidal. 
%\end{proof}

Recall from Definition \ref{def:convex} that each equivalence class for a convex pre-order $\succ$ is a face. We call a module $\succ$-cuspidal if it is cuspidal for some equivalence class. 
%The following is immediate from \cite[Theorem 2.19]{tingleywebster}.

\begin{theorem} \label{classification}  \cite[Theorem 2.19]{tingleywebster}
If $L_1,\dots,L_h$ are simple, $\succ$-cuspidal, and satisfy $\wt(L_1)\succ
  \cdots \succ
  \wt(L_h)$, then $L_1 \circ \cdots \circ L_h$
  has a unique simple quotient.   Furthermore, every simple appears this way up to a grading shift for a unique
  sequence of simple cuspidal representations.
  Here $\wt(L_i) \succ \wt(L_j)$ means that these are in the positive spans of some equivalence classes $C_i, C_j$, and $C_i \succ C_j$.% {\color{blue} added this last sentence} % \qed
\end{theorem}

\begin{corollary}\cite[Corollary 2.17]{tingleywebster}
Fix a convex total order $\succ$ and a positive real root $\a$. There is a unique self-dual simple cuspidal $R(\a)$-module $L(\a)$.
\end{corollary}

%Fix simple $\succ$-cuspidal representations $L_1,\dots,L_h$ with $L_1$ Assume each $L_i)Then $L_1 \circ \cdots \circ L_h$
%  has a unique simple quotient. 
%    Furthermore, every simple appears this way up to a grading shift for a unique
%  sequence of simple cuspidal representations.

\begin{definition}
Let $\alpha$ be a positive real root. Let $F(\a)$ be the face defined by 
$$F^+ = \{\beta \mid \beta \succ \alpha\}, \quad F= \{ \alpha \}, \quad F^- = \{\beta \mid \beta \prec \alpha\}.$$
Define the root module $\D(\a)$ to be the projective cover of $L(\a)$ in $\C_{F(\a)}$. 
\end{definition} 

%{\color{orange} I added the middle sentence, addressing point \S 9}

%%%%%%%%%%%%%%%%%%%%%%%%%%%%%%%%%%%%%%%%%%%%%%%%%%%%%%%%%%%%%%%%
\subsection{Crystal structures} \label{ss:cryst-def}
%%%%%%%%%%%%%%%%%%%%%%%%%%%%%%%%%%%%%%%%%%%%%%%%%%%%%%%%%%%%%%%%%%%%

\begin{definition}\label{def:bbc}
Let $\B$ be the set of simple self-dual elements of $R$-mod. For each face $F$ let $\underline \B_{F}$ be the set of self-dual simple elements of the cuspidal category $\C_F$. 
\end{definition}

\begin{theorem} \cite{laudavazirani}\label{thm:lvcrystal}
$\B$ is isomorphic to a copy of $B(\infty)$ for the root system $\Phi$, where the crystal operators are given by
$f_i(L) = q_i^{\epsilon_i(L)} \hd(L\circ L(i))$. 
\end{theorem}

\begin{remark} \label{rem:gs1}
The precise grading shift is not given in \cite{laudavazirani}. It can be found without proof in \cite[Definition 4.16(iii)]{KKKO2}. For completeness here is a proof: By the definition of $\epsilon_i$, % It is immediate from the results in \cite{laudavazirani} that
$$\Res_{ \wt(L)-\epsilon_i(L) i, \epsilon_i(L) i } L \simeq L' \otimes  L(\epsilon_i(L) i)   \  \text{and} \  \Res_{ \wt(L)-(\epsilon_i(L)+1) i, (\epsilon_i(L)+1) i } L =0$$
for a self-dual $L'$. By rank 1 calculations, $ L(\epsilon_i(L) i) \circ L(i)   = q_i^{-\epsilon_i(L)} L((\epsilon_i(L)+1) i)$. It follows that
$$\Res_{ \wt(L)-\epsilon_i(L) i, (\epsilon_i(L)+1) i, } L \circ L(i) \simeq q_i^{-\epsilon_i(L)}  L' \otimes L((\epsilon_i(L)+1) i).$$
Since the restriction of a self-dual representation is self-dual the statement follows. 
\end{remark}

\begin{theorem} \cite{tingleywebster}  \label{th:tingleywebster} 
Assume $\Phi$ is at worst affine. Then:
%{\color{blue} also requires at the full KLR algebra is at worst affine}
\begin{itemize}
\item If $F$ is finite type, $\underline \B_{F}$ is isomorphic to a copy of $B(\infty)$ for the root system $\Phi_F$. The crystal operators are given by, for each $\underline i \in \Delta_F$, 
$f_{\underline{i}}(L) = q_{\underline{i}}^{\epsilon_{\underline{i}}(L)}\hd(L\circ L(\underline{i}))$. 

\item If $\Phi_F$ has one or more affine components, $\underline \B_F$ is isomorphic to a union of infinitely many copies of $B(\infty)$ for $\Phi_F$, with crystal operators 
$f_{\underline i}(L) =  q_{\underline{i}}^{\epsilon_{\underline{i}}(L)}  \hd (L\circ L(\underline i))$. 
\end{itemize}
\end{theorem}

\begin{remark}
The grading shift is not specified in \cite{tingleywebster}, but follows by an an argument similar to Remark \ref{rem:gs1}. The key step is that, for some self-dual $L'$,
$$\Res_{R(\wt(L)- \epsilon_{\underline{i}}(L) \underline{i}), R(\underline{i})^{\otimes \epsilon_{\underline{i}}(L)+1}} L \circ L(\underline{i}) \simeq q_{\underline i}^{-\epsilon_{\underline{i}}(L)} [\epsilon_{\underline{i}}(L)+1]! L' \otimes  L({\underline i})^{\otimes \epsilon_{\underline{i}}(L)+1}.$$
\end{remark}

%\begin{remark}
%In \cite{tingleywebster} the exact grading shift in $f_{\underline{i}}$ is not given. Its value follows from 
%Theorem \ref{thm:lvcrystal} and the existence of the face functors developed in this paper.
%{\color{blue} does this follow if the KLR algebra was not geometric type?}
%{\color{orange} i'm not sure yet (need to think). i suspect there is a general argument involving extremal words that appear in the character of an irrep} {\color{blue} Well, at very least it doesn't follow from the existance of face functors unless those face functors exist...but face functors exist for all symmetric affine faces of all affine faces, since the only on-geometric ones are type A, and then all roots in the face root system have finite type support. But I'm not sure we know this for all non-symmetric affine cases.   } {\color{orange} yes it doesn't follow from the face functors since they don't always exist in this generality}
%{\color{blue} Also, can we give a precise reference to where the grading shift is addressed in LV?}
%\end{remark}

%%%%%%%%%%%%%%%%%%%%%%%%%%%%%%%%%%%%%%%%%%%%%%%%%%%%%%%%%%
\section{The face cuspidal category} \label{sec:cuspcat}
%%%%%%%%%%%%%%%%%%%%%%%%%%%%%%%%%%%%%%%%%%%%%%%%%%%%%%%%000

Fix a root system $\Phi$ and let $R$ be an associated KLR algebra. Fix a face $F$ and a compatible convex order $\prec$ (which is possible by Lemma \ref{lem:order-exists}). 
Let $\C$ be the full subcategory of $R$-mod consisting of $F$-cuspidal representations (which is $\C_F$ from \S\ref{ss:cuspidal}).
The main purpose of this section is to define a projective object $P$ in $\C$, and, under certain conditions, show that $\End(P)$ is isomorphic to a KLR algebra for the root system $\Phi_F$.
\subsection{Properties}

\begin{lemma} \label{lem:closedcirc}
 $\C$ is abelian and is closed under $\circ$. That is, if $M,N \in \C$, then $M \circ N \in \C$.
\end{lemma}

\begin{proof}
It is clear that $\C$ is abelian.
Fix $M,N$ and choose $c$ as in \S\ref{ss:cuspidal} for $\wt(M)+\wt(N)$. 
By the definition of $\circ,$ if $e_\ii (M \circ N) \neq 0$, then $\ii$ is a shuffle of some $\ii_1$ and $\ii_2$ such that $e_{\ii_1} M \neq 0$ and $e_{\ii_2} N \neq 0$. Thus, if $\ii'$ is a prefix of $\ii$, then $\ii'$ is a shuffle of a prefix $\ii'_1$ of $\ii_1$ and $\ii_2'$ of $\ii_2$. But then $c(\ii_1'), c(\ii_2') \leq 0$, so $c(\ii')\leq0$. This holds for all $\ii$ such that $e_\ii (M \circ N) \neq 0$ and all prefixes, so $M \circ N \in \C$ by definition.    
\end{proof}

\begin{lemma}\label{independent}
For any 
 $\a\in\D_F$ the modules $L(\a)$ and $\D(\a)$ depend only on the face, not on the choice of compatible convex order.
\end{lemma}

\begin{proof}
Fix convex orders $\prec, \prec'$ compatible with $F$. Let $L(\a)$ be the simple cuspidal for $\a$ with respect to $\prec$. Consider the cuspidal decomposition 
$$L(\a) = \hd( L_1' \circ \cdots \circ L_k')$$
from Theorem \ref{classification} with respect to $\prec'$. If $\wt(L_1')$ is not in $F$, then 
$$\Res_{\wt(L'_1), \a-\wt(L_1')} L(\a) \neq 0$$ contradicts the cuspidality of $L(\a)$ with respect to $\prec$, and similarly for $L_n'$. Thus, for all $k$, $\wt(L_k')=s_k \b_k$ for $\b_k \in F$. But 
$\a= \sum s_k \b_k$ contradicts $\a \in \Delta_F$ unless the cuspidal decomposition is only one step. Thus $L(\a)$ is cuspidal for $\prec'$. Since $L(\a)$ is the unique simple of weight $\alpha$ in $\mathcal{C}$ defined using either $\prec$ or $\prec'$, it follows that $\Delta(\alpha)$ also does not depend on this choice.  
%
% By the cuspidality of $L(\a)$ with respect to 
%
%Because $\a$ is a simple root for $\Phi_F$, every way of writing $\a=\la+\mu$ with $\la,\mu\in\N I$ both nonzero, $\la\in \spann_{{\Bbb R}_{\geq 0}}(F\cup F^+)$, and $\mu\in \spann_{{\Bbb R}_{\geq 0}} (F^-\cup F)$, actually has $\la\in \spann_{{\Bbb R}_{\geq 0}}(F^+)$ and $\mu\in\spann_{{\Bbb R}_{\geq 0}} (F^-)$.
%Thus, for all convex orders $\prec$ compatible with the face, the condition for a $R(\a)$-module to be $\prec$-cuspidal can be expressed purely in terms of the face $(F^+,F,F^-)$, proving the lemma.
%{\color{blue} I don't understand this proof. If the spans were over ${\Bbb Z}_{\geq 0}$ then it would be fine, but over ${\Bbb R}_{\geq 0}$, why can't $\alpha$ involve something from $F$? I think maybe we need to talk about cuspidal decompositions, which force the coefficients to be integers}{\color{orange}use the partial order on $\N I$ where $x<y$ if $y-x\in\N I$. $\la$ can't contain anything from $F$ since there are no elements of $F$ less than $\a$.}
\end{proof}

\begin{lemma} \label{lem:cuspidal-filtration}  
Let $\beta_1, \ldots, \beta_n, \gamma_1, \ldots \gamma_n \in \Delta_F$ with $\sum \b_i=\sum \ga_j$. Let $M_1, \ldots, M_n$ be face-cuspidal modules with $\wt(M_i)=\beta_i$. Consider
\[ \Res_{\gamma_1, \ldots \gamma_n} (M_1 \circ \cdots \circ M_n).\]
The non-zero layers in the Mackey filtration from Theorem \ref{mackey} are exactly those corresponding to permutations of the factors (i.e. where for some permutation $\sigma$, $\eta_{ij} =0$ unless $j=\sigma(i)$).
The sub-quotients are of the form $  M_{\sigma^{-1}(1)} \otimes \cdots \otimes M_{\sigma^{-1}(n)},$ with a shift by the degree of $\sigma$ acting on the idempotent $e_{\beta_1 \cdots \beta_n}.$ 
In particular, the restriction is non-zero if and only if $\{ \beta_1, \ldots, \beta_n\} = \{ \gamma_1, \ldots, \gamma_n\}$ as multisets.
\end{lemma}

\begin{proof}
It is clear that there is a subquotient $\Res_{\gamma_1, \ldots \gamma_n} (M_1 \circ \cdots \circ M_n)$ in the Mackey filtration corresponding to any permutation $\sigma \in S_n$ such that $\ga_{\sigma(i)}=\beta_i$ for all $i$. 
Suppose there is a non-zero subquotient that does not correspond to such a permutation. Let $k$ be the first index such that $\ga_k$ is a nontrivial sum $\ga_k=\sum_i \eta_{ik}$. For each $i$ such that $\eta_{ik}\neq 0$, the minimality of $k$ implies that $\Res_{\eta_{ik},\b_i-\eta_{ik}} M_i\neq 0$. The face-cuspidality of $M_i$ implies that $c(\eta_{ik})\leq 0$, 
where $c$ is a functional compatible with $F$ as in Lemma \ref{lem:co-hyper} for $N= \sum \b_i$.
Since $\ga_k$ is a simple root in the face $F$, $\eta_{ik}$ does not lie in $F$ and hence, since $\eta_{ik}$ is a prefix of $\beta_i$, $c(\eta_{ik})<0$. Since $\ga_k=\sum_i \eta_{ik}$, convexity implies that $c(\ga_k)<0$, a contradiction.
%
%
%
%Let $_\ga S_\b$ be the set of permutations $\sigma \in S_n$ such that $\ga_{\sigma(i)}=\beta_i$ for all $i$. 
%It is clear that there are subquotients of $\Res_{\gamma_1, \ldots \gamma_n} (M_1 \circ \cdots \circ M_n)$ in the Mackey filtration (Theorem \ref{mackey}) corresponding to permutations in $_\ga S_\b$.
%Suppose we have a subquotient that does not correspond to a permutation in ${}_\ga S_\b$. Let $k$ be the first index such that $\ga_k$ is a nontrivial sum $\ga_k=\sum_i \eta_{ik}$. For each $i$ such that $\eta_{ik}\neq 0$, the minimality of $k$ implies that $\Res_{\eta_{ik},\b_i-\eta_{ik}} M_i\neq 0$ if we are to get a nonzero piece of the filtration. The face-cuspidality of $M_i$ implies that $c(\eta_{ik})\leq 0$, 
%where $c$ is a functional compatible with $F$ in the sense of the paragraph following Lemma \ref{lem:co-hyper}.
%Since $\ga_k$ is a simple root in the face $F$, $\eta_{ik}$ does not lie in $F$ and hence $c(\eta_{ik})<0$. Since $\ga_k=\sum_i \eta_{ik}$, convexity implies that $c(\ga_k)<0$, a contradiction.
\end{proof}

\begin{lemma}\label{resm} Fix  $\underline i_1,\ldots,\underline i_n \in \D_F$ and $M \in \C$. Every simple subquotient of $\Res_{\underline i_1,\ldots,\underline i_n}M$ is of the form $L(\underline i_1)\otimes \cdots\otimes L(\underline i_n)$ up to a grading shift.
\end{lemma}

\begin{proof} 
By Lemma \ref{lem:co-hyper}, we can choose a linear functional $c$ such that $F \subset c^{-1}(0)$ and, for any root $\gamma$ of height at most $\text{ht}(\wt(M))$, $\gamma \in F^-$ if and only if $c(\beta)<0$, and $\beta \in F^+$ if and only if $c(\beta)>0$. 

Let 
 $M_1 \otimes \cdots \otimes M_n,$ be a simple subquotient of $\Res_{\underline i_1,\ldots,\underline i_n}M$. Assume that, for some $j$, $M_j$ is not isomorphic to $L(\underline i_j)$, and take $j$ minimal for this to occur. Let $M_j = \hd (N_1\circ \ldots \circ N_k)$ be the decomposition of $M_j$ as in Theorem \ref{classification} for the corresponding convex preorder, and let $\lambda= \underline i_1+ \cdots + \underline i_{j-1} + \wt(N_1)$. 
Since $
\underline i_j$ is simple for $\Phi_F$,
$\wt(N_1)$ cannot be a multiple of a root in $F$, so $c(N_1) > c(M_j)=0$. But then 
$$e_{\lambda, \wt(M)-\lambda} M \neq 0 \qquad 
\text{ and } 
\qquad c(\lambda)= 0 + \cdots+0+c(\wt(N_1))>0,$$
By Lemma \ref{lemma:other-cuspidal}, this contradicts the assumption that $M$ was cuspidal.
%{\color{orange}I changed the last sentence here to explicitly reference 2.19.}
\end{proof}

\subsection{The projective $P$} \label{ss:defP}

\begin{theorem}\label{projective}
Let $\underline i_1,\ldots,\underline i_n\in \Delta_F$. Then $\D(\underline i_1)\circ\cdots\circ \D(\underline i_n)$ is projective in $\C$. 
\end{theorem}

\begin{proof} 
By Lemma \ref{lem:closedcirc}, $\D(\underline i_1)\circ\cdots\circ \D(\underline i_n) \in \C$. To show projectivity, it is equivalent to show that \[
\Ext^1(\D(\underline i_1)\circ \cdots \circ \D(\underline i_n),M)=0
\] for all $M \in \C$.
By the adjunction \eqref{extadjunction}, this is equivalent to showing 
\[
\Ext^1(\D(\underline i_1)\otimes \cdots \otimes \D(\underline i_n),\Res_{\underline i_1,\ldots,\underline i_n}M)=0.
\]
By Lemma \ref{resm}, it suffices to prove 
\[
\Ext^1(\D(\underline i_1)\otimes \cdots \otimes \D(\underline i_n),L(\underline i_1)\otimes \cdots\otimes L(\underline i_n))=0.
\]
Since $\Ext^1(\D(\underline i_k),L(\underline i_k))=0$, this is true.
\end{proof}

\begin{definition}
Let $P$ be the direct sum of all modules of the form $\D(\underline i_1)\circ \cdots\circ\D(\underline i_n)$ for $\underline i_1,\ldots,\underline i_n\in\D_F$. 
The indexing set for these direct summands is $ \coprod_\nu \Seq_F(\nu)$ where 
\[
\Seq_F(\nu) = \{ (\underline i_1, \ldots, \underline i_n) : \underline i_k \in \Delta_F, \underline i_1+\cdots +\underline i_n=\nu \}.
\]
For $\underline{\ii}=(\underline i_1,\ldots,\underline i_n)\in \Seq_F(\nu)$, we write $\D(\underline{\ii})$ for $\D(\underline i_1)\circ\cdots\circ\D(\underline i_n)$.
\end{definition}

\begin{theorem}\label{projgenerator} 
% {\color{blue} may require assumption that the whole KLR algebra is at worst affine}
If $\Phi$ is of finite or affine type and $\Phi_F$ is of finite type, then $P$, together with its grading shifts, is a projective generator of $\C$.
\end{theorem}

\begin{proof}
By Theorem \ref{projective}, $\D(\underline i_1)\circ\cdots\circ \D(\underline i_n)$ is projective. It suffices to show that every simple module $L$ in $\C$ is a quotient of a module of the form $\D(\underline i_1)\circ\cdots\circ\D(\underline i_n)$. 
By \cite[Corollary 3.29]{tingleywebster}, every non-trivial simple in $\C$ is of the form
$ \hd (L(\underline i) \circ X)$
for some $\underline i \in \Delta_F$ and some simple $X$ in $\C$, so by induction is a quotient of a module of the form $\D(\underline i_1)\circ\cdots\circ\D(\underline i_n)$. 
\end{proof}

\begin{remark}
A counterexample to Theorem \ref{projgenerator} when $\Phi$ and $\Phi_F$ are both of affine type is given in \S\ref{sec:example2}. A counterexample when $\Phi$ is of hyperbolic type and $\Phi_F$ is finite type is provided in \cite[\S 3.7]{tingleywebster}. There $\Phi_F$ is of type $\mathfrak{sl}_2$, but there are also imaginary cuspidal modules, contradicting the statement. 
% {\color{blue} I'm not sure this is a counter example, and in fact, if you believe in reflection functors, maybe there shouldn't be a counter example. I'm not sure..In the example from my paper with Ben, the face involved in not finite type, as it involves imaginary roots. However, the .}
\end{remark}

Fix $\underline \ii =(\underline i_1, \cdots, \underline i_n), \underline \jj = (\uj_1, \cdots, \uj_n) \in \Seq_F(\nu)$. Let $_{\underline \jj} S_{\underline \ii}$ be the subset of the symmetric group $S_n$ consisting of permutations taking $\underline \ii$ to $\underline \jj$. For each $w\in {_{\underline \jj} S_{\underline \ii}}$, pick a reduced decomposition $w=s_{i_1}\cdots s_{i_n}$. Set
\[
 \tau_w(\underline \ii) = \tau_{i_1}(s_{i_2}\cdots s_{i_n}\underline \ii)\cdots \tau_{i_n}(\underline \ii).
\]
Define $\deg_{\underline \ii}(w)$ to be $\deg \tau_w(\underline \ii)$, which clearly does not depend on the reduced decomposition.

\begin{lemma}\label{restrictdelta}  Fix $\underline \ii, \underline \jj \in \Seq_F(\nu)$. Then
$\displaystyle 
  \Res_{\underline \jj} \D({\underline \ii})\cong\bigoplus_{w\in {{}_{\underline \jj} S_{\underline \ii}}}q^{\deg _\ii(w)} \D(\uj_1)\otimes\cdots\otimes \D(\uj_n).
$
\end{lemma}

\begin{proof}
By Lemma \ref{lem:cuspidal-filtration} there is a filtration of $\Res_{\underline \jj} \D({\underline \ii})$ with these subquotients. 
This filtration splits since $\Ext^1(\D(\uj_i),\D(\uj_i))=0$. 
\end{proof}

\subsection{Generators of $\End(P)$} \label{ss:gens}
Write $P=\oplus_\nu P_\nu$ where $P_\nu$ is a $R(\nu)$-module. 
Given a sequence $\underline \ii=(\underline i_1,\ldots,\underline i_n)$ of elements in $\D_F$, let
\[
 e_{\underline \ii}  \in \End(P)
\]
be the projection onto the summand $\D(\underline i_1)\circ\cdots\circ\D(\underline i_n)$.

For $m,n \in {\Bbb N}$, let $w[m,n]$ be the element of the symmetric group $S_{m+n}$ given by
\[
 w[m,n](i)=\begin{cases}
            i+n \quad\text{if} \quad \!\!\! i\leq m \\
            i-m\quad \text{otherwise}.
           \end{cases}
\]
There is a unique reduced expression $s_{i_1} \cdots s_{i_l}$ for $w[m,n]$ up to two-term braid moves (since it has a unique descent). Thus we can unambiguously define an element $\psi_{ w[m,n]}=\psi_{i_1}\cdots \psi_{i_l} \in R$.

For each positive real root $\a$, fix a nonzero vector $v_\a \in \D(\a)$ of minimal degree.

\begin{lemma} \label{lem:tau}
Fix $\underline i , \underline k \in\D_F$. 
There is a unique homomorphism
 \[
  \tau\map{q^{-\underline i \cdot\underline k}\D(\underline i)\circ\D(\underline k)}{\D(\underline k)\circ\D(\underline i)}
 \]
such that $\tau (1 \otimes (v_{\underline i} \otimes v_{\underline k})) = \psi_{w[\text{ht}(\underline i),\text{ht}(\underline k)]} \otimes (v_{\underline k}  \otimes v_{\underline i})$.
\end{lemma}

\begin{proof}
If $\underline i=\underline k$, this is \cite[Lemma 3.6]{bkm} (there it is assumed that $\Phi$ is finite type, but the same proof works in general). Now assume that $\underline i\neq\underline k$. By adjunction,
 \[
  \Hom(q^{-\underline i \cdot\underline k }\D(\underline i)\circ\D(\underline k),\D(\underline k)\circ\D(\underline i))\cong 
  \Hom( q^{-\underline i \cdot\underline k}\D(\underline i)\otimes \D(\underline k),\Res_{\underline i,\underline k}\D(\underline k)\circ \D(\underline i)).
 \]
 By Lemma \ref{lem:cuspidal-filtration}, the Mackey filtration of $\Res_{\underline i,\underline k}\D(\underline k)\circ \D(\underline i)$ from Theorem \ref{mackey} has a unique nonzero layer, resulting in an isomorphism
\[
 \Res_{\underline i,\underline k}\D(\underline k)\circ \D(\underline i)\cong q^{-\underline i\cdot\underline k}\D(\underline i)\otimes\D(\underline k).
\]
Thus
 \begin{equation*}
  \Hom(q^{-\underline i \cdot\underline k } \D(\underline i)\circ\D(\underline k),\D(\underline k)\circ\D(\underline i))\cong \Hom(q^{-\underline i \cdot\underline k } \D(\underline i)\otimes \D(\underline k),q^{-\underline i\cdot\underline k}\D(\underline i)\otimes\D(\underline k)).
 \end{equation*}
 Tracing the identity map through the isomorphisms gives the desired $\tau$. 
 Uniqueness follows because $1 \otimes (v_{\underline i} \otimes v_{\underline k})$ generates $\D(\underline i)\circ \D(\underline k)$.
\end{proof}

For a sequence $\underline {\bf i}$ of elements in $\Delta_F$, Define
\[
 \tau_k(\underline \ii)\in \End(P).
\]
to act by $\id\circ\cdots\circ\id\circ\tau\circ\id\circ\cdots\circ\id$ on $\D(\underline i_1)\circ\cdots\circ \D(\underline i_n)$,
where the $\tau$ is the homomorphism from Lemma \ref{lem:tau} acting in the $k$-th and $(1+k)$-th place.

For $\alpha \in \Phi$ consider the restriction $R^{\text{sup}(\alpha)}$ of $R$ to the sub-Dynkin diagram where $\alpha$ is supported. From now on we make the following assumption about faces $F$.

\begin{assumption} \label{assumptionB} For every $\alpha \in \Delta_F$, $R^{\text{sup}(\alpha)}$ satisfies Assumption \ref{assumptionA}. 
\end{assumption}

The following is \cite[Theorem 3.3(4)]{bkm} if  $R^{\text{sup}(\alpha)}$ is of finite type, and \cite[Theorem 18.3]{mcn3} if it is of geometric symmetric affine type over a field of characteristic 0, together with \cite[Lemma 3.9]{bkm}.
\begin{theorem} \label{lem:dot}
For each $\a \in \Delta_F$, $\End(\D(\a))\cong k[x_\a]$, where $x_\a$ has degree $\a \cdot \a$, and $\D(\a)$ is a free $ k[x_\a]$-module. 
Furthermore, there is a unique such isomorphism satisfying the following relations in $\End(\D(\a)\circ\D(\a))$:
 $\tau x_2 = x_1 \tau+1$ and $x_2\tau=\tau x_1+1$, where $x_1$ and $x_2$ are the endomorphisms $x_\alpha\circ \id$ and $\id\circ x_\alpha$ respectively.
\qed
\end{theorem}

%\begin{theorem} \label{th:deg-of-end} 
%Assumption \ref{assumptionB} holds if the KLR algebra $R$ is either of finite type, or of geometric symmetric affine type over a field of characteristic zero. That is, Assumption \ref{assumptionA} implies Assumption \ref{assumptionB}.
%\end{theorem}
%
%\begin{lemma}\label{lem:ass}
%Assumption \ref{assumptionB} holds whenever each root $\a\in\D_F$ has support which is either a finite type subdiagram or a symmetric affine subdiagram when the field is of characteristic zero and the KLR algebra is geometric.
%\end{lemma}
%
%\begin{proof}
%Under the stated assumption, $\Delta(\a)$ is a root module in a sub-KLR algebra where Assumption \ref{assumptionA} holds, and hence by Theorem \ref{th:deg-of-end} Assumption \ref{assumptionB} holds for $\Delta(\a)$. This is true for all $\a \in \Delta_F$, so Assumption \ref{assumptionB} holds. 
%\end{proof}
%
%\begin{remark}
%We believe Assumption \ref{assumptionB} holds in even greater generality, which would imply that our face functor construction works in that generality as well.  
%\end{remark}
%

\begin{definition} \label{def:gens} $ x_k(\ii)  \in \End(P)$ is the element which acts 
 by $\id\circ\cdots\circ\id\circ x \circ\id\cdots\circ\id$ on $\D(\underline i_1)\circ\cdots\circ \D(\underline i_n)$ and by zero on all other summands, where $x$ is the element from Lemma \ref{lem:dot} acting on $\Delta(\underline i_k)$.
\end{definition}

\begin{theorem}\label{basisthm} 
 Let ${\underline \ii},\underline \jj\in \Seq_F(\nu)$. Then %, under Assumption \ref{assumptionB},
 \[
  \{ \tau_w(\underline \ii) x_1(\underline \ii)^{a_1}\cdots x_{n}(\underline \ii)^{a_{n}}\mid w\in {_{\underline \jj} S_{\underline \ii}}, a_1,\ldots,a_{n}\in \N\}
 \] 
is a basis of $\Hom(\D(\underline \jj),\D(\underline \ii))$.
\end{theorem}

\begin{proof}
By Theorem \ref{lem:dot} each $\Delta(\underline i_k)$ is a free module for $k[x_k]$. Using this and Lemma \ref{lem:tau}, 
the endomorphisms in the statement produce linearly independent vectors when applied to the element $v_{\underline i_1}\otimes\cdots\otimes v_{\underline i_{n}}$, where as above $v_{\underline i}$ is a chosen vector in $\Delta(\underline i)$ of minimal degree. Hence they are linearly independent. 

By adjunction and Lemma \ref{restrictdelta},
 \begin{align*}
  \Hom(\D(\underline \jj),\D(\underline \ii))&\cong \Hom(\D(\uj_1)\otimes\cdots\otimes \D(\uj_n),\Res_{\underline \jj}\D(\underline \ii)) \\
  &\cong \bigoplus_{w\in {_{\underline \jj} S_{\underline \ii}}}q^{\deg \tau_w(\underline \ii)} \End(\D(\uj_1)\otimes\cdots\otimes \D(\uj_n)) \\
  &\cong \bigoplus_{w\in {_{\underline \jj} S_{\underline \ii}}}q^{\deg \tau_w(\underline \ii)} \bigotimes_{i=1}^n\End(\D(\uj_i)).
 \end{align*} 
Therefore
\[
 \dim_q\Hom(\D(\underline \jj),\D(\underline \ii)) 
 = \sum_{w\in_{\underline \jj} S_{\underline \ii}}  q^{\deg \tau_w(\underline \ii)} \prod_{i=1}^n(1-q^{\uj_i\cdot\uj_i})\inv.
\]
This dimension count shows that our linearly independent set is a basis.
\end{proof}

\begin{corollary}\label{kxy} If $\underline i$ and $\underline k$ are distinct roots in $\D_F$, then 
 \[
  \End(\D(\underline i)\circ \D(\underline k))\cong k[x_{\underline i},x_{\underline k}].
 \] \qed
\end{corollary}
%where $x_{\underline i}$ and $x_{\underline k}$ are from Assumption \ref{assumptionB}. 

\subsection{Relations in $\End(P)$}
For distinct $\underline i,\uj\in\D_F$, define $Q_{\underline i \uj}(u,v)\in k[u,v]$ by
\begin{equation}\label{def:q}
\tau_1(\uj \underline i) \tau_1(\underline i \uj) = Q_{\underline i \uj} (x_1(\underline i \uj),x_2(\underline i \uj)).
\end{equation}
These exist by Corollary \ref{kxy} and are homogeneous of degree $-2 \, \underline i\cdot \uj$ by Theorem \ref{lem:dot}. For $\underline i=\uj$, define $Q_{\underline i \underline i}=0$. Then $\tau_1(\underline i \underline i) \tau_1(\underline i \underline i) = Q_{\underline i\underline i} (x_1(\underline i \underline i),x_2(\underline i \underline i))=0$, since by Theorem \ref{basisthm} the degree $-2 \underline i \cdot \underline i$ Hom space is zero.

\begin{lemma}\label{easyrelation} 
 If $\underline i\neq \uj$, then $x_1\tau_1(\underline i \uj)=\tau_1(\underline i \uj)x_2$ and $x_2\tau_1(\underline i \uj)=\tau_1(\underline i \uj)x_1$.
\end{lemma}

\begin{proof} 
The module $\D(\uj)$ is generated by $v_{\uj}$ so there exists $a \in R$ such that $xv_{\uj}=av_{\uj}$. We now compute 
\[
\begin{aligned}
 (x_1\tau_1 - \tau_1 x_2)(v_{\underline i} \otimes v_{\uj}) &=  x_1 \psi_{w[\text{ht}(\uj) ,\text{ht}(\underline i)]} v_{\uj} \otimes v_{\underline i} -
 \tau_1 (1 \otimes a) v_{\underline i} \otimes v_{\uj}
 \\
 &=  \psi_{w[\text{ht}(\uj) ,\text{ht}(\underline i)]} x_1 v_{\uj} \otimes v_{\underline i} -
(1 \otimes a)   \tau_1 v_{\underline i} \otimes v_{\uj}
\\
&= (\psi_{w[\text{ht}(\uj),\text{ht}(\underline i)]}(a\otimes 1)-(1\otimes a)\psi_{w[\text{ht}(\uj),\text{ht}(\underline i)]})v_{\uj}\otimes v_{\underline i}.
 \end{aligned}
\]
Apply the straightening relations in the KLR algebra to put $\psi_{w[\text{ht}(\uj),\text{ht}(\underline i)]}(a\otimes 1)-(1\otimes a)\psi_{w[\text{ht}(\uj),\text{ht}(\underline i)]}$ in the standard form as a sum of elements of the form $\psi_w P$ where $w$ is a permutation and $P$ is a polynomial in the $y_i$. Each term that appears has $w$ not greater than or equal to $w[\text{ht}(\uj),\text{ht}(\underline i)]$ in Bruhat order.

% Apply $x_1\tau_1 - \tau_1 x_2$ to $v_{\underline i} \otimes v_{\uj}$.
% By the straightening relations in the KLR algebra, only terms of the form  $\psi_w P ( v_{\uj}\otimes v_{\underline i})$ can appear, where $P$ is a polynomial in the $y_i$ and $w$ is a permutation strictly less than $w[\text{ht}(\uj),\text{ht}(\underline i)]$ in Bruhat order. 
On the other hand $x_1\tau_1 - \tau_1 x_2$ is a module homomorphism, so Theorem \ref{basisthm} shows that 
$$x_1\tau_1 - \tau_1 x_2 =  \psi_{w[\text{ht}(\uj),\text{ht}(\underline i)]} (Ax_1+Bx_2)( v_{\uj}\otimes v_{\underline i})$$
%
%all terms that appear must be of the form 
%$ \psi_w Q ( v_{\uj}\otimes v_{\underline i})$, 
for some $A,B \in k$, giving a contradiction unless $A=B=0$.
\end{proof}

\begin{lemma}
$Q_{\underline i \uj}(u,v)=Q_{\uj \underline i}(v,u)$.
\end{lemma}
\begin{proof} 
By the definition of $ Q_{\underline i \uj}$, 
$$\tau_1(\underline i \uj)\tau_1(\uj \underline i) \tau_1(\underline i\uj)=\tau_1(\underline i \uj)Q_{\underline i\uj}(x_1,x_2).$$
 On the other hand, using the definition of $Q_{ \uj \underline i}$ and the relation in Lemma \ref{easyrelation},
 \begin{align*}
  \tau_1(\underline i \uj)\tau_1(\uj\underline i) \tau_1(\underline i\uj)&= Q_{\uj\underline i}(x_1,x_2)\tau_1(\underline i\uj) \\
  &= \tau_1(\underline i\uj)Q_{\uj\underline i}(x_2,x_1).
 \end{align*}
The result follows by using the basis of 
$\Hom( \D(\underline i)\circ \D(\uj),  \D(\uj)\circ \D(\underline i))$ from Theorem \ref{basisthm}.
\end{proof}

\begin{lemma} \label{lem:ispklr}
The following relations hold: 
\[
\tau_1(\uj\underline k\underline i)\tau_2(\uj\underline i\underline k)\tau_1(\underline i\uj\underline k)-\tau_2(\underline k\underline i\uj)\tau_1(\underline i\underline k\uj)\tau_2(\underline i\uj\underline k)=\begin{cases}
                                                                     \frac{Q_{\underline i\uj}(x_1,x_2)-Q_{\underline i\uj}(x_3,x_2)}{x_1-x_3} \quad & \text{if} \ \underline i=\underline k \\
                                                                     0 &\text{otherwise.}
                                                                    \end{cases}
\]
\end{lemma}

\begin{proof}
Apply $\tau_1\tau_2\tau_1-\tau_2\tau_1\tau_2$ to $v_{\underline i }\otimes v_{\uj}\otimes v_{\underline k}$.
By the straightening relations in the KLR algebra, only terms of the form $T \psi_w ( v_{\underline k}\otimes v_{\uj}\otimes v_{\underline i} )$ can appear, where $T$ is a polynomial in the $y_i$ and $w$ is a permutation strictly less than $w_1 w_2 w_1$ in Bruhat order. Here $w_i$ is the permutation that acts as $w[\text{ht}(\underline i_k ),\text{ht}(\underline i_{{k+1}})]$ on the $i, i+1$ factors, and trivially on the other factor. 

On the other hand $\tau_1\tau_2\tau_1-\tau_2\tau_1\tau_2$ is a module homomorphism, so Theorem \ref{basisthm} severely restricts its possibilities: By the observations of the previous paragraph, this difference must be a composition of at most two $\tau_i$ followed by some $x_i$. If $\underline i$, $\uj$ and $\underline k$ are all distinct, this difference must be zero. If $\underline i=\uj\neq \underline k$, the only terms which perform the right permutation on the factors are of the form $\tau_1\tau_2$ composed with some $x_i$, but these are all of higher degree, so the difference again is zero. The $\underline i\neq \uj=\underline k$ case is similar. 
 If $\underline i=\uj=\underline k$, there are no possibilities of the right degree (i.e. $-3 i \cdot i$), so the difference is zero. 
 
It remains to consider $\underline i=\underline k\neq \uj$.
Considerations as above show that $\tau_1\tau_2\tau_1-\tau_2\tau_1\tau_2=P(x_1,x_2,x_3)$ for some polynomial $P$. Multiplying on the left by $\tau_1$ gives
 \[
  \tau_1^2\tau_2\tau_1-\tau_1\tau_2\tau_1\tau_2=\tau_1 P.
 \]
Apply the case of this Lemma which we have already proven to get
\[
 \tau_1^2\tau_2\tau_1-\tau_2\tau_1\tau_2\tau_2=\tau_1 P.
\]
Simplifying the $\tau_i^2$ terms gives
\begin{equation} \label{eq:stt}
 Q_{\uj \underline k} (x_1,x_2)\tau_2\tau_1-\tau_2\tau_1 Q_{\uj \underline k}(x_2,x_3)=\tau_1 P.
\end{equation}
Applying the relations from Lemma \ref{easyrelation} (pushing a dot past a crossing in diagrammatic notation) and using 
 Theorem \ref{basisthm}, \eqref{eq:stt} uniquely determines $P$. But the stated $P$ does satisfy the equation, since this is true in the standard KLR algebra. 
\end{proof}

\begin{proposition} \label{prop:is-pKLR}
$\End(P)$ is a pseudo-KLR algebra.
\end{proposition}

\begin{remark}
We will prove that $\End(P)$ is actually a KLR algebra in Theorem \ref{4.12}.
\end{remark}

\begin{proof}
 Let $A$ be the pseudo-KLR algebra defined using the root system $\Phi_F$ and the polynomials $Q_{i'j'}(u,v)$ defined in (\ref{def:q}) for $i',j' \in \Delta_F$. The above results show that there is a homomorphism from $A$ to $\End(P)$. This is surjective by Theorem \ref{basisthm}. To show it is an isomorphism, it suffices to show that $\dimq A \leq \dimq \End(P)$. The latter dimension is known by Theorem \ref{basisthm}. The former dimension is bounded above by this since we can use the usual straightening relations to put each element of $A$ in a standard form.
\end{proof}

\begin{lemma}\label{numberofquotients}
 Let $A$ be a $\Z$-graded algebra with $\dim A_n$ finite for all $n$ and zero for sufficiently negative $n$.
 Let $P$ be a finitely generated projective $A$-module and let $B=\End_A(P)$. Then the number of irreducible modules for $B$ 
 is equal to the number of irreducible quotients of $P$, where both counts are taken up to grading shift and isomorphism.
\end{lemma}

\begin{proof}
 Write $P=\oplus_i P_i^{\oplus d_i}$ where the $P_i$ are pairwise non-isomorphic indecomposable projectives and the $d_i$ are positive integers. Then $\End(P)=\oplus_{i,j} \Hom(P_i^{d_i},P_j^{d_j})$. The radical of $\End(P)$ is generated by the radicals of each $\End(P_i^{d_i})$ as well as all of $\Hom(P_i,P_j)$ for $i\neq j$. The maximal semisimple quotient of $\End(P)$ is $\oplus_i \End(\hd(P_i)^{\oplus d_i})$, which is a direct sum of matrix algebras. Therefore the number of irreducible representations of $\End(P)$ is equal to the number of irreducible quotients of $P$.
\end{proof}

\begin{lemma}\label{nonzerocoeff} For all $\underline i \neq \uj \in \Delta_F$, 
the coefficient of $u^{-a_{\underline i \uj}}$ in $Q_{\underline i \uj}(u,v)$ is nonzero. Here $a_{\underline i, \uj}$ is the entry in the Cartan matrix for $\Phi_F$. 
\end{lemma}

\begin{proof}
Let $P_{(1-a_{\underline i \uj})\underline i+\uj}$ be the component of $P$ of weight $(1-a_{\underline i \uj})\underline i+\uj$.
By Proposition \ref{prop:is-pKLR}, $\End P_{(1-a_{\underline i \uj})\underline i+\uj}$ is a pseudo-KLR algebra. 
Suppose instead that the coefficient of $u^{-a_{\underline i \uj}}$ in $Q_{\underline i \uj}(u,v)$ was zero. 
By Theorem \ref{pseudo} and Lemma \ref{numberofquotients}, $P_{(1-a_{\underline i \uj})\underline i+\uj}$ must have at least $2-a_{\underline i \uj}$ pairwise non-isomorphic irreducible quotients.

Pick a convex order $\prec$ on $\Phi^+$ such that $\underline i$ and $\uj$ span a compatible face. If the face is of finite type then by Theorem \ref{th:tingleywebster} the number of cuspidal modules for this face is the number of elements of the corresponding rank two crystal, which is $1-a_{\underline i \uj}$, so, this is a contradiction.

Since $\Phi$ is either of finite or symmetric affine type, it remains to consider the case where $\underline i$ and $\uj$ span a root system of type $\widehat{\mathfrak{sl}}_2$. Then the minimal imaginary root is $\d=\underline i+\uj$, and $P_{3\underline i+\uj}$ has at least four pairwise non-isomorphic irreducible quotients. Without loss of generality, $\uj \succ \underline i$. In terms of the classification in Theorem \ref{classification}, at least two of them are of the form $\hd (L \circ L(2 \underline i))$ where $L$ is a cuspidal $R(\d)$-module.
 
Since $P_{3\underline i+\uj}$ is projective, there is a nonzero morphism
\[
 P_{3\underline i+\uj} \to L\circ L(2 \underline i) \to L\circ L( \underline i)  \circ L( \underline i).
\]
The restriction-coinduction adjunction \eqref{otheradjunction} gives a corresponding nonzero morphism
\[
 \Res_{\underline i,\underline i,\d} P_{3\underline i+\uj} \to L(\underline i)\otimes L(\underline i)\otimes L.
\]
The same argument as in the proof of Lemma \ref{lem:cuspidal-filtration} shows that $\Res_{\underline i,\underline i,\d} P_{3\underline i+\uj}$ is a direct sum of modules of the form 
\[
 \D(\underline i)\otimes \D(\underline i) \otimes (\D(\underline i)\circ \D(\uj))\ \mbox{  or  }\  \D(\underline i)\otimes \D(\underline i) \otimes (\D(\uj)\circ \D(\underline i)).
\]
By the cuspidality of $L$, there is no non-zero morphism from $\D(\uj)\circ \D(\underline i)$ to $L$. Therefore there must be a nonzero morphism from $\D(\underline i)\circ \D(\uj)$ to $L$.

By Corollary \ref{kxy} the algebra $\End(\D(\underline i)\circ \D(\uj))$ is isomorphic to $k[x,y]$, which has a unique irreducible quotient, so by Lemma \ref{numberofquotients}, $\D(\underline i)\circ \D(\uj)$ has a unique irreducible quotient. But we have two non-isomorphic modules $L$ which are both quotients of $\D(\underline i)\circ \D(\uj)$, so this is a contradiction.
\end{proof}

\begin{theorem}\label{4.12}
The algebra $\End(P)$ is a KLR algebra for the root system $\Phi_F$.
\end{theorem}

\begin{proof}
This is immediate from Proposition \ref{prop:is-pKLR} and Lemma \ref{nonzerocoeff}.
\end{proof}

%%%%%%%%%%%%%%%%%%%%%%%%%%%%%%%%%%%%%%%%%%%
\section{Face functors} %%%%%%%%%%%%%%%%%%%
%%%%%%%%%%%%%%%%%%%%%%%%%%%%%%%%%%%%%%

\subsection{Definition}  \label{sec:ff}
Fix a face $F$. Let $R_F$ be the graded algebra $\End(P)^{\op}$, where $P$ is the projective object in $\C_F$ from \S\ref{ss:defP}.
Theorem \ref{4.12} shows that $R_F^\op$ is a KLR algebra for the root system $\Phi_F$. The existence of the anti-automorphism $\dagger$ (see \S\ref{ss:KLR}) implies that $R_F$ is an isomorphic KLR algebra.

\begin{definition} 
 The face functor is the functor $\F : X \rightarrow P \otimes_{R_F} X$
from $R_F$-mod to $\C_F$.
\end{definition}

The image of $\F$ lies in $\C_F$ since $\F$ sends a free rank one $R_F$-module to $P$, $\F$ is right exact and $\C_F$ is closed under taking quotients.

\begin{remark}
In finite type, this face functor is an example of the functors constructed by Kashiwara and Park \cite{kashiwarapark} in terms of duality data (see \cite[Proposition 3.5]{kashiwarapark}). In affine type, we expect that this is also the case, but we do not have a proof. 
\end{remark}

\subsection{Categorical properties}

\begin{theorem} \label{th:feq}
If $\Phi$ is of finite or affine type and $F$ is of finite type, then $\F$ is an equivalence of categories. 
\end{theorem}

\begin{proof}
 This is immediate from Theorem \ref{projgenerator} which shows that $P$ is a projective generator.
\end{proof}

\begin{lemma}\label{fullyfaithful}
 For any face $F$, the face functor $\F$ is fully faithful.
\end{lemma}

\begin{proof}
We use the criterion that a left adjoint is fully faithful if and only if the unit of adjunction is a natural isomorphism. The functor $\F$ is left adjoint to $\Hom(P,-)$. The unit of this adjunction is an isomorphism on free modules by the construction of the face KLR algebra as an endomorphism algebra. The general case follows by considering a free resolution since tensoring is right exact and the Hom functor is exact since $P$ is projective.
\end{proof}

\begin{lemma}\label{commutewithind}
For all  $R_F$ modules $A$ and $B$ there is a natural isomorphism
\[
 \G(A)\circ\G(B)\cong \G(A\circ B).
\]
\end{lemma}

\begin{proof}
Write $R_F(\la)$ for $\End(P_\la)^{\op}$. Let 
%$e_{\la\mu}$ be the image of the unit under the inclusion $R(\la)\otimes R(\mu)\to R(\la+\mu)$ and 
$e^F_{\la\mu}$ the image of the unit under the inclusion $R_F(\la)\otimes R_F(\mu)\to R_F(\la+\mu)$. The two functors $A\otimes B\mapsto \F(A\circ B)$ and $A\otimes B\mapsto \F(A)\circ \F(B)$ from $R_F(\la)\otimes R_F(\mu)\mods$ to $R(\la+\mu)\mods$ are given by tensoring with the bimodules
\[
 P_{\la+\mu} \bigotimes_{R_F(\la+\mu)} R_F(\la+\mu)e^F_{\la\mu}
\]
and
\[
 R(\la+\mu)e_{\la\mu} \bigotimes_{R(\la)\otimes R(\mu)} (P_\la\otimes P_\mu)
\]
respectively.
These are both canonically isomorphic to the direct sum over all
$\underline i_1,\ldots,\underline i_n\in\D_F$ with $\underline i_1+\cdots+\underline i_l=\la$ and $\underline i_{l+1}+\cdots+\underline i_n=\mu$ of 
$$\D(\underline i_1)\circ\cdots \circ \D(\underline i_n).$$
Therefore the functors are equivalent.
\end{proof}

\subsection{Compatibility with standard modules}\label{sub:standard}
Results in this section hold under assumption \ref{assumptionA}, which gives us access to the theory of standard modules introduced in \cite{bkm} and \cite{mcn3}. These depend on the convex order $\prec$ and are built out of root modules. The root modules corresponding to real roots have already been introduced, these are the modules $\D(\a)$. For the indivisible imaginary root $\d$, we will call the modules denoted $\D(\w)$ in \cite{mcn3} root modules. These are the projective modules in the category of cuspidal $R(\d)$-modules. 

Standard modules are naturally indexed by root partitions. A root partition is a sequence $\la=(\a_1^{n_1},\cdots,\a_l^{n_l})$ where $\a_1\succ \cdots \succ \a_l$ are indivisible roots, each $n_i$ is a positive integer unless $\a_i=\delta$, in which case it is a collection of partitions. To each term $\a_i^{n_i}$ a standard module $\D(\a_i)^{(n_i)}$ is constructed. 
If $\a_i$ is real then $\D(\a_i)^{\circ n_i}$ is a direct sum of $n_i!$ copies of the module $\D(\a_i)^{(n_i)}$ with grading shifts.
If $\a_i$ is imaginary then $\D(\a_i)^{(n_i)}$ is a summand of a product of certain modules $\D(\w)$ of weight $\delta$ in $\C_F$; see \cite{mcn3} for the details (where this module is denoted $\D(\underline{\la})$). The standard module is then defined to be the indecomposable module
\[
 \D(\la)=\D(\a_1)^{(n_1)}\circ\cdots\circ \D(\a_l)^{(n_l)}.
\]
In \cite{bkm} and \cite{mcn3} homological properties of these modules are developed which justify the name standard in the setting of affine quasi-hereditary algebras.

\begin{definition}
Fix a convex order $\prec$ and a root $\alpha$. A minimal pair for $\a$ is an ordered pair of roots $(\b,\ga)$ such that $\a=\b+\ga$, $\ga\prec\b$, and there is no pair of roots $(\b',\ga')$ satisfying $\a=\b'+\ga'$ and $\ga\prec\ga'\prec\b'\prec\b$.
\end{definition}

\begin{lemma} \label{lem:set}
Fix a convex order $\prec$. Let $\D$ be a root module with $\wt(\D)=\a$ and let $(\b,\ga)$ be a minimal pair for $\a$. 
Then there exist root modules $\D_\b$ and $\D_\ga$ with $\wt(\D_\b)=\b$ and $\wt(\D_\ga)=\ga$ such that there is a short exact sequence
\[
  0\to q^{-\b\cdot\ga} \D_\b \circ \D_\ga \xrightarrow{f_{\b\ga}} \D_\ga\circ\D_\b\to \D^{\oplus m}\to 0
 \] for some nonzero $m\in\N[q,q\inv]$. Furthermore the homomorphism $f_{\b\ga}$ spans the degree zero part of $\Hom(q^{-\b\cdot\ga} \D_\b \circ \D_\ga,\D_\ga\circ\D_\b)$.
 
 Conversely suppose that $\D_\b$ and $\D_\ga$ are root modules with $\wt(\D_\b)=\b$ and $\wt(\D_\ga)=\ga$. Then the degree zero part of $\Hom(q^{-\b\cdot\ga} \D_\b \circ \D_\ga,\D_\ga\circ\D_\b)$ is one dimensional. If $f$ is any nonzero degree zero homomorphism, then $f$ is injective and $\coker{f}$ is a direct sum of root modules. 
\end{lemma}

\begin{remark}
The notation e.g. $\Delta_\b,$ as opposed to $\Delta(\b)$, is because, if $\b=\delta$, then $\D_{\delta}$ can be any of the modules $\D(\omega)$ of weight $\delta$. 
\end{remark}

\begin{proof}
 In finite type, this is \cite[Theorem 4.10]{bkm}, and in symmetric affine type it is \cite[Lemma 16.1]{mcn3} for real roots and \cite[Theorem 17.1]{mcn3} for imaginary roots. The statement about the space of degree zero homomorphisms 
being one-dimensional is not explicitly mentioned, but is clear from the proofs.
\end{proof}

Given a face $F$ and a convex order $\prec$ compatible with $F$, we naturally get a convex order $\prec_F$ on the face root system $\Phi_F$. Thus we can talk about root modules for both $R$ and $R_F$.

\begin{lemma} \label{lem:roottoroot}
The face functor $\G$ sends root modules to root modules.
\end{lemma}

\begin{proof}
Let $\D$ be a root module for $R_F$ and let $\a=\wt(\D)$.
Proceed by induction on the height of $\a$, the case $\wt(\D)\in\D_F$ being trivial. If $\alpha \not\in\D_F$ there is a minimal pair $(\b,\ga)$ for $\a$, and by Lemma \ref{lem:set} there is a short exact sequence 
 \[
  0\to q^{-\b\cdot\ga} \D_\b \circ \D_\ga \to \D_\ga\circ\D_\b \to \D^{\oplus m}\to 0
 \]
 where $\D_\b$ and $\D_\ga$ are root modules with $\wt(\D_\b)=\b$ and $\wt(\D_\ga)=\ga$.

By Lemma \ref{commutewithind} and the fact that $\G$ is right exact, the following is exact:
\[
 q^{-\b\cdot\ga}\G(\D_\b)\circ\G(\D_\ga)\to \G(\D_\ga)\circ\G(\D_\b)\to \G(\D)^{\oplus m}\to 0.
\]
By the inductive hypothesis, $\G(\D_\b)$ and $\G(\D_\ga)$ are root modules. Since $(\b,\ga)$ is a minimal pair, Lemma \ref{lem:set} implies that $\G(\D)^{\oplus m}$ is a direct sum of root modules. 
Since root modules are indecomposable, the Krull-Schmidt theorem implies that $\G(\D)$ is a root module, as required. 
\end{proof}

\begin{proposition}\label{prop:stdtostd}
 The face functor $\G$ takes standard modules to standard modules.
\end{proposition}

\begin{proof}
Standard modules are built from root modules from a process of inducing and taking direct summands. By Lemma~\ref{lem:roottoroot}, $\G$ takes root modules to root modules. The functor $\G$ commutes with induction by Lemma \ref{commutewithind}. Since $\F$ is fully faithful it takes indecomposables to indecomposables, and clearly commutes with taking direct sums. \end{proof}

\subsection{Compatibility with nesting}

\begin{proposition} \label{prop:comp} 
Let $E \subset F$ be nested faces, both satisfying Assumption \ref{assumptionB}. Assume further that $E$ satisfies Assumption \ref{assumptionB} with respect to the KLR algebra $R_F$. 
Then $(R_F)_E$ is isomorphic to $R_E$. Furthermore,
$\F_E$ and $ \F_F \circ  \F^F_E $ are naturally isomorphic, where $\F^F_E: R_E\text{-mod} \rightarrow R_F\text{-mod}$ is the face functor for $E$ considered as a face of $\Phi_F$.  
\end{proposition}

\begin{proof}
Let $P_F$, $P_E$ and $P^F_E$ be the modules used to define the functors
$\F_F$ and $ \F_E$ and $\F^F_E$ respectively. Then
\begin{align*}
P_E=\bigoplus_{\ii\in \Seq_{E}} \D(i_1)\circ\cdots \circ \D(i_n) \quad \text{and} \quad
P^F_E=\bigoplus_{\ii\in\Seq_{E}} \D_F(i_1)\circ\cdots \circ \D_F(i_n),
\end{align*}
where $\D(\a)$ and $\D_F(\a)$ refer to the root modules for the KLR algebras $R$ and $R_F$ respectively.
For any $i \in \Delta_E$, $R^{\text{sup}(i)}$ satisfies Assumption \ref{assumptionA}, so,
by Lemma \ref{lem:roottoroot}, $\F_F(\Delta_F(i))= \Delta(i)$. By Lemma \ref{commutewithind} we see $\F_F(P_E^F)=P_E$, so by Lemma \ref{fullyfaithful} $\End(P_E^F)\cong\End(P_E)$, and hence $(R_F)_E \simeq R_E$ by definition. 

Since $\F_F(P_E^F)=P_E$ and $R_{E}\cong \End(P_E)^{\op}\cong\End(P^F_E)^{\op}$, we obtain an isomorphism of $(R,R_{E})$-bimodules
\[
 P^F_E \otimes_{R_F} P_F \cong P_E.
\]
This completes the proof.
\end{proof}

\subsection{Compatibility with crystal operators}\label{sub:crystal}
In this section we work under assumption \ref{assumptionA}.
Recall $\underline \B_F$ from Definition \ref{def:bbc}.
Let $\B_F$ be the set of self-dual simple modules for the KLR algebra $R_F$.
Recall also the crystal operators on $\underline \B_F$  and $\B_F$ from \S\ref{ss:cryst-def}. Theorem \ref{th:feq} immediately implies that, in finite type, the face functor gives a bijection $\B_F \mapsto \underline \B_F$, and this intertwines the crystal operators for $\B_F$ and the face crystal operators on $\underline \B_F$. We now prove a weaker version of this that holds in affine type (see Corollary \ref{cor:fcrystal}).

\begin{lemma}\label{simplehead}
If $L$ is simple then $\G(L)$ has simple head.
\end{lemma}

\begin{proof}
Every simple module $L$ is the head of a standard module $\D$. As $\G$ is right exact, $\G(L)$ is a quotient of $\G(\D)$. But $\G(\D)$ is standard by Proposition \ref{prop:stdtostd}, so has a simple head. By Lemma \ref{fullyfaithful}, $\G(L)\neq 0$, so this completes the proof. %which is true by Lemma \ref{fullyfaithful}.
\end{proof}

\begin{proposition} \label{prop:gta}
Let $ \Delta'_F$ be the set of simple roots of $\Phi_F$ thought of as its own root system, not as a sub-root system of $\Delta$. For all simple modules $L$ of $R_F$ and $i' \in \Delta'_F$,
 \[
  \hd (\F(\hd (L \circ L(i'))))\cong \hd (\hd\F(L) \circ  \F(L(i'))),
 \] up to a grading shift, and this is a nonzero simple module.
\end{proposition}

\begin{remark}
 We expect that the grading shift in Proposition \ref{prop:gta} is unnecessary.
\end{remark}

\begin{proof}
Since $\F$ is right exact, $\F(L\circ L(i'))$ surjects onto $\F(\hd (L \circ L(i')))$ and hence onto $\hd (\G(\hd (L \circ L(i'))))$. This is simple by Lemma \ref{simplehead}, since by \cite[Lemma 3.9]{khovanovlauda} $\hd (L \circ L(i'))$ is simple.
 
 On the other hand, by Lemma \ref{commutewithind} we have $\G(L\circ L(i'))\cong \G(L)\circ \G(L(i'))$. This surjects onto $\hd \G(L)\circ \G(L(i'))$ and hence onto $\hd (\hd \G(L) \circ \G(L(i')))$. This last is simple by \cite[Proposition 3.22]{tingleywebster} since $\G$ sends $L(i')$ to a simple cuspidal module. Thus both sides of the equation are simple quotients of $\F(L\circ L(i'))$. 
 
Choose a convex order on $\Phi_F$ with $i'$ minimal, and a compatible convex order on $\Phi$.  
 The module $L\circ L(i')$ is a quotient of $\D(L) \circ \D(i')$, and it is immediate from the definition of standard modules (see \cite[\S24]{mcn3}) that this last is isomorphic to a direct sum of copies of $\D({f}_{i'}(L))$. 
Here $\D(L)$, $\D( f_{i'}(L))$ and $\D(i')$ are the standard modules corresponding to $L$, $ f_{i'}( L)$ and $L(i')$. Since $\G$ sends standard modules to standard modules, we see that $\G(L\circ L(i'))$ can only have one isomorphism class of simple quotient up to a grading shift, which completes the proof. 
\end{proof}

\begin{corollary} \label{cor:fcrystal} 
The injection $\B_F \rightarrow \underline{\B}_F$ defined by $L \mapsto \hd \F(L)$ intertwines the crystal operators for $\B_F$ and the face crystal operators on $\underline \B_F$ from \cite{tingleywebster}. \qed
\end{corollary}

\begin{remark} If $F$ is of finite type then, by
\cite[Theorem 4.4(ii)(a)]{kashiwarapark}, $\G$ takes a simple module either to a simple module or zero (their construction is more general, and by Lemma \ref{fullyfaithful} the latter is not possible in our case). This is not true in affine type, as we discuss in \S\ref{sec:example2}. Thus taking the head is necessary in Corollary \ref{cor:fcrystal}.

\end{remark}

\subsection{Example} \label{sec:example}

Let $\Phi$ be a root system of type $A_2^{(1)}$ with simple roots $\a_0$, $\a_1$ and $\a_2$. Fix the polynomials defining the KLR algebra to be $Q_{i,i+1}(u,v)=s_{i}u+t_{i}v$ where all indices are read modulo 3. Let $\pi$ be the standard projection from affine roots to finite type roots which sends $\delta$ to zero. Consider
$$\begin{aligned}
 F^+&=  \pi^{-1} \{  \a_2, \alpha_1+\alpha_2\}, \\
 F&=  \pi^{-1} \{ \alpha_1, -\alpha_1, 0\}, \\
 F^-&=  \pi^{-1} \{ -\alpha_2, -\alpha_1-\alpha_2 \}. 
\end{aligned} $$
Then $(F^+,F,F^-)$ is a face according to Definition \ref{facedefinition}. Notice that $F= \{ m \delta, m \delta \pm \alpha_1 \}$, so $\Phi_F$ is of type $\widehat{\mathfrak{sl}}_2$ with simple roots $\b=\a_1$ and $\ga=\a_0+\a_2$. The support of both these roots is finite type, so Assumption \ref{assumptionB} holds, regardless of where the original KLR algebra was of geometric type, and the face functor is defined. 
 
The cuspidal $R(\ga)$-module has character $[02]$.
  Let $v_{\ga}$ be a lowest degree element of $\D(\ga)$. The endomorphism $x_\ga$ of $\D(\ga)$ (normalized as in Theorem \ref{lem:dot}) satisfies
 \[
  x_\ga v_{\ga} =-s_2\inv y_1 v_{\ga} = t_2\inv y_2 v_{\ga}.
 \]
Explicit computation shows that 
 \[
  Q_{\b\ga}(u,v)= s_1t_0 u^2 +(t_0t_1t_2-s_0s_1s_2)uv-s_0s_2t_1t_2 v^2.
 \]
Note that
\begin{itemize}
 \item The coefficient of $uv$ is zero if and only if $s_0s_1s_2=t_0t_1t_2$.
 \item The discriminant of $Q$ is zero if and only if $s_0s_1s_2+t_0t_1t_2=0$.
\end{itemize}
These observations imply that
\begin{itemize}
 \item Examples 3.3 and 3.4 of \cite{kashiwara} are related by a face functor.
 \item By Lemma \ref{lem:geom-characterization}, $R$ is of geometric type if and only if $R_F$ is.
\end{itemize}
%In Theorem \ref{th:geom-to-geom} below we will show that, in general, if $R$ is of geometric type then, for any face $F$, $R_F$ is also of geometric type. 

\section{Imaginary cuspidal representations and affine faces} \label{sec:imcat}

For this section $\Phi$ is of affine type and, unless otherwise stated, the KLR algebra $R$ is of geometric type. Fix a convex order $\prec$ on $\Phi^+$. Let $\d$ be the minimal imaginary root. We study the category of cuspidal $R(\d)$-modules. 
In particular, we show that the endomorphism algebra of a projective generator of the category of cuspidal $R(\d)$-modules is isomorphic to $k[z]\otimes Z$ where $Z$ is the zigzag algebra from \cite{HK} corresponding to the underlying finite type Dynkin diagram. We show this by using face functors to reduce to the $ \widehat{\mathfrak{sl}}_3$ case.
We also show that, if $R$ is of geometric type, so is any face KLR algebra $R_F$ (see Theorem \ref{th:geom-to-geom}).

Let $S(\delta)$ be the quotient of $R(\d)$ by the two sided ideal generated by all $e_{\bf i}$ where $\ii$ has a proper prefix $\ii'$ with $\wt(\ii')$ a sum of roots $\succ \delta$. Then the category of $S(\delta)$-modules is equivalent to the category of cuspidal $R(\delta)$-modules, so this agrees with the algebra $S(\d)$ from \cite[\S 12]{mcn3}. By \cite[\S 17]{mcn3}, the simple $S(\d)$-modules are naturally parametrized by a set $\Omega$ of chamber coweights for an underlying finite type Cartan matrix. For each $\w\in\Omega$, let $L(\w)$ be the self-dual irreducible module parametrized by $\w$ and let $\D(\w)$ be its projective cover in the category of cuspidal $R(\d)$-modules. Additionally, to each $\w\in \Omega$, consider the positive real roots $(\w_-, \w_+)$, defined in \cite[\S 12]{mcn3} (see also \cite[\S3.4]{tingleywebster}, where $\w_-$ is called $\beta_0$) which have the following properties:
\begin{itemize}

\item $\w_- \in F^-$ and $\w_+ \in F^+$, 

\item $L(\w) = \hd (L(w_-) \circ L({w_+}))$, \text{ and}

\item $\{ \pi(\w_+) : \w \in \Omega\}$ is a positive system for the underlying finite type root system, where $\pi$ is the standard projection from \S\ref{sec:example}.
\end{itemize}

For each positive real root $\a$ let $E_\a$ and $E_\a^*$ be the PBW and dual PBW basis vectors for the PBW basis associated to $\prec$, as defined in \cite[\S 9]{mcn3} (see also \cite{BCP, BN}). 
When the Grothendieck group of $R$-modules is identified with $U_q^+(\mathfrak{g})$ we have \cite[Theorems 9.1 and 18.2]{mcn3} which say that $[L(\a)]=E_\a^*$ and $[\D(\a)]=E_\a$.

\subsection{The Cartan matrix for $S(\d)$}

The following is immediate from \cite[Proposition 3.31 and Lemma 3.44]{tingleywebster}, but for completeness we provide an alternative proof. 

\begin{lemma} \label{lem:xw-}
For all $x \neq \w$ in $\Omega$, we have $L(x)\circ L(\w_-)\cong L(\w_-)\circ L(x)$.
\end{lemma}

\begin{proof}
By Theorem \ref{classification} there is a short exact sequence
\[
0 \to X \to L(x)\circ L(\w_-) \to L(x,\w_-)\to 0,
\]
where $L(x,\w_-)$ is the irreducible module associated to the root partition $(x,\w_-)$ and, by \cite[Theorem 10.1(3)]{mcn3}, $X$ is a successive extension of grading shifts of modules whose cuspidal decomposition only involves roots between $\w_-$ and $\delta$. Since $\pi(\w_-)$ is a simple root for the positive system  $\{ \pi(\w_+) : \w \in \Omega\}$, there is no way to write $\delta+\w_-$ as a non-trivial sum of roots $\prec \delta$, so $X$ is in fact a successive extension of grading shifts of $L(\delta+\w_-)$. By \cite[Lemma 7.5]{mcn3} and \cite[Theorem 9.1]{mcn3}, we have $[X]\in q\N[q]E_{\d+\w_-}^*$ in the Grothendieck group of $R$-modules, identified  with $U_q^+(\mathfrak{g})$.

By (\ref{otheradjunction}), 
\[
\Hom(L(\d+\w_-),L(x)\circ L(\w_-)) \cong \Hom(\Res_{\w_-,\d}L(\d+\w_-),L(\w_-)\otimes L(x)).
\]
The proof of \cite[Lemma 21.8]{mcn3} shows that $\Res_{\w_-,\d}L(\d+\w_-)$ has $L(\w_-)\otimes L(x)$ appearing with graded multiplicity 1 or 0. Therefore the only possible homomorphisms between $L(\d+\w_-)$ and $L(x)\circ L(\w_-)$ are in degree zero. Combined with $[X]\in q\N[q]E_{\d+\w_-}^*$, the only option is $X=0$.

Therefore $L(x)\circ L(\w_-)\cong L(x,\w_-)$. Taking duals gives $ L(\w_-)\circ L(x) \cong L(x,\w_-)$, and the Lemma follows. 
\end{proof}

\begin{lemma}\cite[Lemma 21.7]{mcn3} \label{lem:M21.7}
 For any $\w \in \Omega$ there is a short exact sequence
 \[
  0\to qL(\d+\w_-)\to L(\w)\circ L(\w_-)\to L(\w,\w_-)\to 0.
 \]
\end{lemma}

By Lemmas \ref{lem:M21.7} and \ref{lem:xw-}, 
\begin{equation*}
[L(x)]E_{\w_-}^*=\begin{cases} [L(\w,\w_-)]+q E^*_{\d+\w_-} &\text{if } x=\w \\
E^*_{\w_-}[L(x)] &\text{otherwise.}
\end{cases}
\end{equation*}
In the first case, apply the bar involution and use (\ref{eq:dualind}) to obtain
\[
 E_{\w_-}^*[L(\w)]=[L(\w,\w_-)]+q\inv E^*_{\d+\w_-}.
\]
Thus we have
the commutator formula
\begin{equation}\label{commutator1}
[E_{\w_-}^*,[L(x)]]=\begin{cases} (q^{-1}-q)E^*_{\d+\w_-} &\text{if } x=\w \\
0 &\text{otherwise.}
\end{cases}
\end{equation}

\begin{lemma}
 Let $x$ and $\w$ be two chamber coweights. Then
\begin{equation}\label{commutator2}
[E_{\w_-},[\D(x)]]=\begin{cases}
(q+q\inv)E_{\d+\w_-} &\text{if $x=\w$}\\
0  &\text{if $x$ and $\w$ are orthogonal} \\
E_{\d+\w_-} &\text{otherwise.}
\end{cases}
\end{equation}
\end{lemma}
\begin{proof}
 The ideas of this proof are in \cite[Lemma 21.8]{mcn3}. Indeed, when $x=\w$, this lemma is the decategorification of the $n=0$ case of \cite[Lemma 21.8]{mcn3}.
 
 Now suppose that $x\neq \w$. By \cite[Lemma 16.1]{mcn3}, there is a short exact sequence
 \[
  0\to \D(x)\circ \D(\w_-)\to \D(\w_-)\circ \D(x) \to C\to 0
 \]
 where $C$ is a direct sum of $f(q)$ copies of $\D(\d+\w_-)$ for some $f(q)\in\N[q,q\inv]$. Passing to the Grothendieck group, this decategorifies to
 \begin{equation}\label{eq:c1}
  [E_{\w_-},[\D(x)]]=f(q)E_{\d+\w_-}.
 \end{equation}
 
 When specialized to $q=1$, $E_{\w_-}$ and $E_{\d+\w_-}$ become the root vectors $e_{w_-}$ and $e_{w_-}\otimes t$ respectively in the Lie algebra and $[\D(x)]$ becomes $h_x\otimes t$ by \cite[Corollary 17.2]{mcn3}. There is
 \begin{equation}\label{eq:c2}
  [e_{\w_-},h_x\otimes t]= \langle x,\w_-\rangle e_{\w_-} \otimes t.
 \end{equation}
 
 Since $f(q)\in\N[q,q\inv]$, equations (\ref{eq:c1}) and (\ref{eq:c2}) complete the proof when $x$ and $\w$ are orthogonal. In the one remaining case, they imply that $f(q)=q^n$ for some $n\in \Z$. Now compute
 \begin{align*}
f(q)&= \langle [C],[L(\d+\w_-)]\rangle \\
&= \langle E_{\w_-}E_x-E_x E_{\w_-},[L(\d+\w_-)] \rangle \\
&= \langle E_{\w_-}E_x,[L(\d+\w_-)] \rangle \\
&=\langle E_{\w_-}\otimes E_x, [\Res_{\w_-,\d}L(\d+\w_-)] \rangle,
\end{align*}
where the third equality holds by the cuspidality of $L(\d+\w_-)$.
By \cite[Lemma 12.3]{mcn3}, $\Res_{\w_-,\d}L(\d+\w_-)$ is a successive self-extension of grading shifts of simples of the form $L(\w_-)\otimes L(y)$ for $y\in\Omega$. The above equations imply that $f(q)$ is the multiplicity of $L(\w_-)\otimes L(x)$ in the Jordan-Holder filtration of $\Res_{\w_-,\d}L(\d+\w_-)$. Since $L(\d+\w_-)$ is self-dual and restriction commutes with duality, we obtain $f(q)=f(q\inv)$. Since $f(q)=q^n$, this forces $n=0$, completing the proof.
\end{proof}

\begin{proposition}\label{prop:cartanmatrix}
 Fix $x, \w$ in $\Omega$. The multiplicity of $L(x)$ in $\D(\w)$ is 
 \[
[\D(\w):L(x)]=\begin{cases}
\frac{1+q^2}{1-q^2} &\text{if $x=\w$}\\
0  &\text{if $x$ and $\w$ are orthogonal} \\
\frac{q}{1-q^2} &\text{otherwise.}
\end{cases}
\]
\end{proposition}

\begin{proof}
 Write $[\D(y)]=\sum_x p_{y,x}[L(x)]$. By (\ref{commutator1}), 
  \[
  (q\inv-q)p_{y,\w}E_{\d+\w_-}^* = [E_{\w_-}^*,[\D(y)]].
 \]
The formula (\ref{commutator2}) along with $E_\a^*=(1-q^{\a\cdot\a})E_\a$ completes the proof.
\end{proof}

\begin{corollary}\label{cor:cartanmatrix}
 Suppose $x,\w\in\Omega$. Then 
 \[
\dim_q\Hom(\D(\w),\D(x))=\begin{cases}
\frac{1+q^2}{1-q^2} &\text{if $x=\w$}\\
0  &\text{if $x$ and $\w$ are orthogonal} \\
\frac{q}{1-q^2} &\text{otherwise.}
\end{cases}
\]
 \end{corollary}

\subsection{Faces in geometric type}

The following is immediate from  \cite[Lemma 15.1]{mcn3}. 

\begin{lemma} \label{lem:ssi} Assume $R$ has the standard choice of parameters from \cite{mcn3}. 
Let $S'(\d)$ be the subalgebra of $S(\d)$ generated by the $e_\ii$, $y_j-y_{j+1}$ and $\psi_k$. Then
$S(\d)\cong k[z]\otimes S'(\d)$ where $z$ has degree $2$.  \qed
\end{lemma} 

\begin{theorem}\label{enddw}
 Let $\w$ be a chamber coweight. Then $\End(\D(\w))\cong k[z,\epsilon]/(\epsilon^2)$ where $z$ and $\epsilon$ are in degree 2.
\end{theorem}

\begin{proof} We can assume $R$ has the standard choice of parameters from \cite{mcn3}. Then
\begin{align*}
  \End_{R(\d)} \D(\w) & \cong \End_{S(\d)} \D(\w) \\
  &\cong k[z]\otimes \End_{S'(\d)} (\D(\w)/z\D(\w)) \\
  &\cong k[z]\otimes k[\epsilon]/(\epsilon^2).
 \end{align*}
The second isomorphism holds by Lemma \ref{lem:ssi} and the fact that $\D(\w)$ is a projective $S(\d)$-module. The third holds by Corollary \ref{cor:cartanmatrix}.
\end{proof}

\begin{lemma} \label{lem:sl2geom}
If $R$ is a geometric KLR algebra of symmetric affine type and $F$ is a face with root system $\Phi_F$ of type $\widehat{\mathfrak{sl}}_2$, then $R_F$ is of geometric type.
\end{lemma}

\begin{proof}
Consider the face functor $\F: R_F\mods \to R\mods.$
Let $Q_{0'1'}(u,v)$ be the quadratic polynomial defining the KLR algebra $R_{F}$. Use the convex order where $1'\succ 0'$, and let $\D$ be the unique indecomposable projective in the category of cuspidal $R_{F}(\d)$ modules. Direct computation shows $\End(\D)\cong k[u,v]/Q_{0'1'}(u,v)$ and that there is a short exact sequence 
$$ 0 \rightarrow q^2 \Delta(1') \circ \Delta(0') \rightarrow  \Delta(0') \circ \Delta(1') \rightarrow \Delta \rightarrow 0.$$
The argument in the proof of Lemma \ref{lem:roottoroot} then shows that $\F (\D)$ is a root module for $R$ so, by Theorem \ref{enddw}, $\End(\F (\D))\cong k[z,\e]/\e^2$. Since face functors are fully faithful, 
$$k[u,v]/Q_{0'1'}(u,v) \cong k[z,\e]/\e^2,$$
which implies that $Q_{0'1'}(u,v)$ has discriminant zero. Then $R_{F}$ is of geometric type by Lemma \ref{lem:geom-characterization}. 
\end{proof}

Any face of type $\widehat{\mathfrak{sl}}_2$ faces inside $\widehat{\mathfrak{sl}}_n$ satisfies Assumption \ref{assumptionB}, so the face functor is defined, regardless of whether the original KLR algebra was of geometric type.

\begin{lemma} \label{lem:upgeom}
Let $R$ be a KLR algebra of type $\widehat{\mathfrak{sl}}_n$. Let $F$ be the face of type $\widehat{\mathfrak{sl}}_2$ defined by
$$\begin{aligned}
 F^+&=  \pi^{-1} \{  \Phi_0^+ \backslash \{\alpha_1 \} \}, \\
 F&=  \pi^{-1} \{ \alpha_1, -\alpha_1, 0\}, \\
 F^-&=  \pi^{-1} \{ \Phi_0^- \backslash \{-\alpha_1 \} \},
\end{aligned} $$
where $\Phi_0$ is the underlying finite type root system, and $\pi$ is the standard projection sending $\delta $ to 0. If $R_F$ is of geometric type, then $R$ is also of geometric type. 
\end{lemma}

\begin{proof} 
We can take  $I={\Bbb Z}/ n {\Bbb Z}$. Let the polynomials defining $R$ be $Q_{i, i+1}(u,v) = s_i u+t_i v$ for some $s_i, t_i\in k^\times$. 

The simple roots of $\Phi^+_F$ are $\underline 0 = 1$ and $\underline 1 = 2+ \cdots +n$. The corresponding simple cuspidal modules are both one dimensional, with characters $[1]$ and $[n \cdots 32]$ respectively. Let $\ii=(n,\ldots,3,2)$.
By cuspidality, for each $1 < j < n$, $\psi_{n-j}$ acts by zero on $\D(\underline 1)$, so $ \D(\underline 1)$ is a cyclic module over $k[y_1, \ldots, y_{n-1}]$. Since 
$$\psi_{n-j}^2e_\ii=Q_{j+1,j}(y_{n-j},y_{n-j+1})e_\ii=(t_{j}y_{n-j}+s_{j}y_{n-j+1})e_\ii,$$
the elements $t_{j}y_{n-j}$ and $-s_jy_{n-j+1}$ act the same on $\D(\underline 1)$.
Since  $\D(\underline 1)$ is infinite dimensional and finitely generated, some $y_j$ must act non-nilpotently, and the above relation implies all do. Thus
\begin{equation} \label{eq:ty}
 (t_2t_3\cdots t_{n-1})y_1,\quad (-1)^n (s_2s_3\cdots s_{n-1})y_{n-1}\quad \mbox{and}\quad \la x_{\underline{1}}
\end{equation} act the same on $\D(\underline 1)$ for some $\la\in k^\times$, where $x_{\underline 1}$ is as in Theorem \ref{lem:dot}. 
%Consider the endomorphisms $x_{\underline 0}$ and $x_{\underline 1}$ of $\Delta(\underline 0)$ and $\Delta(\underline 1)$ respectively from Theorem \ref{th:deg-of-end}. 

By the definitions of $Q_{\underline 0, \underline 1}$ and $\tau$ and the fact that $\tau$ is a homomorphism, for any $v\in \D(\underline 0)$ and $w\in \D(\underline 1)$,
\begin{align*}
 Q_{\underline 0, \underline 1}(x_{\underline 0}, x_{\underline 1})(v\otimes w) &= \tau_1(\underline 1, \underline 0) \tau_1(\underline 0, \underline 1) (v\otimes w) \\
 &= \tau_1(\underline 1, \underline 0) (\psi_1\psi_2\cdots \psi_{n-1})(w\otimes v) \\
 &= (\psi_1\psi_2\cdots \psi_{n-1}) \tau_1(\underline 1, \underline 0)(w\otimes v) \\
 &= (\psi_1\psi_2\cdots \psi_{n-1})( \psi_{n-1}\psi_{n-2}\cdots \psi_1)(v\otimes w).
 \end{align*}
Each $\psi_k$ acts on strands of colors $1$ and $n-k+1$, so the KLR algebra relations give 
 \begin{align*}
  Q_{\underline 0, \underline 1}(x_{\underline 0}, x_{\underline 1})(v\otimes w) &= Q_{1n}(y_1,y_2)Q_{12}(y_1,y_n) (v\otimes w) \\
 &= (t_ny_1+s_ny_2)(s_1y_1+t_1y_n)(v\otimes w).
\end{align*}
By definition $x_{\underline 0}$ acts on $\Delta(\underline 0)$ as $y_1$, and $x_{\underline 1}$ acts on $\Delta(\underline 1)$ as in \eqref{eq:ty} (with a shift in index of the $y_i$ due to the tensor factor $\D(\underline 0)$ on the left), so, for some $\lambda \in k^\times$, 
%Also
%\[
% (t_2t_3\cdots t_{n-1})y_2,\quad (-1)^n (s_2s_3\cdots s_{n-1})y_n\quad \mbox{and}\quad \la x_{\underline{1}}
%\] act in the same way for some $\la\in k^\times$  
\[
 Q_{\underline 0, \underline 1}(u,v)=\left( t_n u + \frac{\la s_n}{t_2t_3\cdots t_{n-1}}v\right)\left(s_1 u + \frac{\la t_1}{(-1)^n s_2s_3\cdots s_{n-1}}v\right).
\]
This has discriminant zero if and only if
$s_1 s_2 \cdots s_{n}= (-1)^{n} t_1 t_2 \cdots t_{n}$. The lemma follows from  Lemma \ref{lem:geom-characterization}, which characterizes geometric type KLR algebras.
\end{proof}

\begin{theorem} \label{th:geom-to-geom}
Suppose $\Phi$ is symmetric and of either finite or affine type. Fix a KLR algebra $R$ of geometric type for $\Phi$, and, if $\Phi$ is affine, assume $R$ is defined over a field of characteristic $0$. Then, for any face $F$, the KLR algebra $R_F$ is also of geometric type.
\end{theorem}

\begin{proof} 
Without loss of generality, $F$ is irreducible. 
The statement is trivial unless $R_F$ is of type $\widehat{\mathfrak{sl}}_n$ for some $n$, since otherwise by Lemma \ref{lem:geom-characterization} all KLR algebras are geometric. So assume $R_F$ is type $\widehat{\mathfrak{sl}}_n$. 
Consider the standard face $E$ of type $\widehat{\mathfrak{sl}}_2$ inside $F$ (i.e. the sub-face as defined in Lemma \ref{lem:upgeom}) and the face functors
\[
 R_E  \mods\to  R_F  \mods \to R\mods.
\]
These are well defined since the two simple roots of $E$ both have support in a finite type sub-Dynkin-diagram of $\Phi_F$. 
By Proposition \ref{prop:comp} this composition of face functors agrees with the one-step face functor, so by Lemma \ref{lem:sl2geom}, $R_{E}$ is of geometric type. But then, by Lemma \ref{lem:upgeom}, $R_F$ is also of geometric type.
\end{proof}

\subsection{The zigzag algebra}
Let $\Gamma$ be a connected graph (the most important case for us is an ADE Dynkin diagram). The zigzag algebra $Z_\Ga$ associated to $\Ga$, as defined in \cite{HK}, is a graded algebra with basis elements $e_i$ for each vertex $i$, $h_{ij}$ for each ordered pair $i,j$ of vertices with an edge from $i$ to $j$, and $w_i$ for each vertex $i$. The elements $e_i$ are in degree zero, $h_{ij}$ are in degree one and $w_i$ are in degree two. Multiplication is given by
\begin{align*}
 e_ie_j&=\delta_{ij}e_j \\
 e_i h_{jk}&=\d_{ij}h_{jk} \\
 h_{ij}e_k&=\d_{jk}h_{ij} \\
 e_iw_j&=w_je_i=\d_{ij}w_i \\
 h_{ij}h_{kl}&=\d_{jk}\d_{il}w_i,
\end{align*}
and all other products of basis elements are zero.

\begin{example} \label{ex:zzA1}
If $\Gamma$ is a single vertex then $Z_\Gamma = \spann \{ e, w \} \simeq k[w]/w^2,$ where $w$ has degree two. If $\Ga$ has no isolated points then $Z_\Gamma$ is generated by the idempotents $e_i$ and the degree 1 elements $h_{ij}$.
\end{example}

\begin{lemma}\label{lem:sdelta}
Let $R$ be a geometric KLR algebra for $\widehat{\mathfrak{sl}}_3$. Then $S(\d)$ is isomorphic to the tensor product of $k[z]$ and the $\mathfrak{sl}_3$-zigzag algebra, where $z$ has degree two.
\end{lemma}
 
\begin{proof}
 We can take $I=\Z/3,$ and the defining polynomials for $R$ to be $Q_{i,i+1}(u,v)=u-v$. 
By applying an outer automorphism of the Dynkin diagram of $\widehat{\mathfrak{sl}}_3$ we can assume without loss of generality that the two cuspidal simple modules of $R(\d)$ have characters $[012]$ and $[021]$ or $[120]$ and $[210]$. We consider only the first case, since the second follows by a symmetric argument. 

Consider $S'(\d)$ from Lemma \ref{lem:ssi}. 
The irreducible representations are one dimensional, so the Cartan matrix computed in Proposition \ref{prop:cartanmatrix} tells us its dimension. Specifically $e_{012}S'(\d)e_{012}$ has dimension $1+q^2$ and $e_{012}S'(\d)e_{021}$ has dimension $q$.
It is now straightforward to see that $\{ e_{012}, \psi_2 e_{012}, (y_2-y_3)e_{012}, e_{021}, \psi_2 e_{021}, (y_2-y_3)e_{021}\}$ is a basis of $S'(\d)$, and that this is isomorphic to the $A_2$ zigzag algebra.
\end{proof}

\begin{theorem} \label{th:zigzag}
 Consider a geometric KLR algebra of symmetric affine type. Let $\d$ be the smallest imaginary root and fix a convex order $\prec$. The category of cuspidal $R(\d)$-modules is equivalent to the category of modules over the algebra $k[z]\otimes Z$ where $Z$ is the zigzag algebra corresponding to the underlying finite type Dynkin diagram.
\end{theorem}

\begin{proof} 
The direct sum of the modules $\D(\w)$ and their grading shifts is a projective generator of the category of $S(\d)$-modules, so by Morita theory it suffices to show
 \[
  \End\left(\bigoplus_{\w\in\Omega} \D(\w)\right)\cong Z\otimes k[z].
 \]
 For the case of $\widehat{\mathfrak{sl}}_2$ this is immediate from Lemma \ref{lem:ssi}. 
 
Otherwise, Corollary \ref{cor:cartanmatrix} shows the endomorphism algebra has the right dimension. 
Let $x$ and $y$ be two elements of $\Omega$ connected by an edge. By Corollary \ref{cor:cartanmatrix} there is a unique up to scalar nonzero degree 1 morphism $h_{xy}\map{\D(x)}{\D(y)}$. By Theorem \ref{enddw}, for all $x\in \Omega$ there is unique up to scalar nonzero degree 2 morphism $\epsilon_x\map{\D(x)}{\D(x)}$ which squares to zero. If we can show that $h_{xy}\circ h_{yx}$ is a nonzero multiple of $\epsilon_x$, then we can rescale the $h_{xy}$ to ensure that this scalar is one, which will complete the proof.

Choosing a linear function $c$ such that $c(x_+)= c(y_+)=c(\delta)=0$, and $c(\alpha)>0$ for all $\alpha$ such that $p(\alpha)$ is another simple root in the positive system, the construction in \S\ref{ss:co} gives a convex pre-order such that the span of $x_\pm$, $y_\pm$ and $\d$ is a face $F$, and by \cite[Lemma 1.14]{tingleywebster} this can be refined to a convex total order with the same positive system $p(\Phi_\prec^+)$. 
By \cite[Theorem 13.1 and Theorem 17.3]{mcn3}, the category of $S(\d)$-modules only depends on the positive system $p(\Phi_\prec^+)$, so without loss of generality we can use this order. 

Consider the face functor
\[
 R_F\mods \to R\mods.
 \]
This is fully faithful and sends standard modules to standard modules, so, to prove that $h_{xy}\circ h_{yx}$ is a nonzero multiple of $\epsilon_x$, it suffices to prove this for $R_F$. But $R_F$ is of type $\widehat{\mathfrak{sl}}_3$ and by Theorem \ref{th:geom-to-geom} is of geometric type, so this is immediate from 
Lemma \ref{lem:sdelta}.
\end{proof}

\begin{remark}
Kleshchev and Muth \cite{KM} have independently proven Theorem \ref{th:zigzag} for balanced convex orders. Their methods rely on case by case computations which we have avoided. They also discuss analogues for  $R(n\d)$ with $n >1$, where they show that the category of cuspidal modules is equivalent to the category of modules for an ``affine zig-zag algebra." In \cite{CL} the affine zig-zag algebra (with some minor sign differences) appears in the categorification of the Heisenberg algebra (as $\text{End}_{{\mathcal{H}}^\Gamma}(P^n)$). Together, these results could perhaps be interpreted as a type of imaginary face functor, with the Heisenberg algebra playing the role of the Kac-Moody (or Borcherds) algebra of the face.
%{\color{blue} I added this reference to Cautis-Licata. I don't really want to elaborate on it, but it seems to me the connection is important}
%from the KLR algebr face functor corresponding to embedding of the Heisenberg algebra into $U_q^{-1(\mthfark{g})$., which is the Borchards algebra generated by . 
\end{remark}

\subsection{Example} \label{sec:example2} Continue the example from \S\ref{sec:example}, which considers a face $F$ of type $\widehat{\mathfrak{sl}}_2$ inside $\widehat{\mathfrak{sl}}_3$. Let $0',1'$, $\delta'$ denote the simple roots and the minimal imaginary root for $\widehat{\mathfrak{sl}}_2$. Choose the convex order with $0'\prec 1'$  and let $\Delta(\delta')$ be the projective cover of $L(\delta')$ in the cuspidal category. Then $L(\delta')$ is one dimensional. There is only one indecomposable projective, so the graded dimension of 
$\Delta(\delta')$ agrees with that of $S(\delta'),$
which by Lemma \ref{lem:ssi} is $\frac{1+q^2}{1-q^2}$. Thus there is an exact sequence 
\begin{equation} \label{eq:ashortexact}
2 q^2 \Delta(\delta') \rightarrow \Delta(\delta') \rightarrow L(\delta') \rightarrow 0.
\end{equation}

Choose a convex order on $\Phi^+$ compatible with the $\widehat{\mathfrak{sl}}_2$ convex order on $F$. The corresponding chamber coweights are the fundamental coweights $\omega_1$ and $\omega_2$.
 Let $\F$ be the face functor associated to $F$ which sends $\D(0')$ to $\D(\a_2+\a_0)$, so that $\F(\D(\delta'))= \D(\omega_2)$. Since $\F$ is right exact, Lemma \ref{lem:roottoroot} implies that \eqref{eq:ashortexact} is sent to an exact sequence
$$2 q^2 \Delta(\omega_2) \rightarrow  \Delta(\omega_2) \rightarrow \F(L(\delta')) \rightarrow 0.$$

 By Theorem \ref{th:zigzag} (or Lemma \ref{lem:sdelta} in this case) the category of cuspidal $R(\delta)$ modules is equivalent to the category of $Z_\Gamma \otimes k[z]$ modules. The zig-zag algebra $Z_\Gamma = \text{span} \{ e_1,e_2, h_{12}, h_{21}, w_1, w_2 \} $ has two simples $S_1,S_2$, and the corresponding projective covers are of the form $Z_\Gamma e_1$ and $Z_\Gamma e_2$. Then $h_{12} e_2$ corresponds to a copy of $q S_1$ in $Z_\Gamma e_2$. Under the equivalence of categories, $Z_\Gamma e_2$ is sent to $\Delta(\omega_2)$ and $S_1$ to $L(\omega_1)$, so we see that $\Delta(\omega_2)$ has a copy of $L(\omega_1)$ in degree $1$. This cannot be in the image of $2 q^2 \Delta(\omega_2)$, so it survives in $ \F(L(\delta'))$, and hence $\F(L(\delta'))$ is not simple. 

Here $\G$ is not an equivalence of categories and Theorem \ref{projgenerator} does not hold as stated for faces not of finite type.

\end{document}